\documentclass[a4paper, 12pt]{amsart}

\usepackage{xcolor}
\usepackage{todonotes}
\usepackage{amsmath, amsthm, amssymb, amsxtra}
\usepackage{amsfonts}
\usepackage[utf8]{inputenc}
\usepackage[dvips]{epsfig}
\usepackage{graphicx}
\usepackage[english]{babel}
\usepackage{hyperref}

\usepackage{graphicx}
\usepackage{latexsym}
\usepackage{mathtools}
\usepackage{esint}
\usepackage{float}

\theoremstyle{plain}
\newtheorem{thm}{Theorem}[section]

\newtheorem{lem}[thm]{Lemma}
\newtheorem{prop}[thm]{Proposition}

\theoremstyle{definition}
\newtheorem{defn}[thm]{Definition}

\theoremstyle{remark}

\numberwithin{equation}{section}

\newcommand{\average}{{\mathchoice {\kern1ex\vcenter{\hrule height.4pt
				width 6pt depth0pt} \kern-9.7pt} {\kern1ex\vcenter{\hrule
				height.4pt width 4.3pt depth0pt} \kern-7pt} {} {} }}

\newcommand{\R}{\mathbb R}

\renewcommand{\L}{\mathcal L}

\newcommand{\p}{\partial}

\newcommand{\comment}[1]{}

\begin{document}
	
    \title[Boundary estimates for parabolic equations in \texorpdfstring{$C^1$}{C1} domains]{Boundary estimates for parabolic non-divergence equations in \texorpdfstring{$C^1$}{C1} domains}
    \author{P\^edra D. S. Andrade}
    \address{Department of Mathematics,\newline
        Paris Lodron Universit\"at Salzburg, \newline
        Hellbrunnerstrasse 34, A-5020 Salzburg, Austria.}
    \email{\tt pedra.andrade@plus.ac.at}
    
    \author{Clara Torres-Latorre}
    \address{Instituto de Ciencias Matemáticas\newline Consejo Superior de Investigaciones Científicas\newline
        C/ Nicolás Cabrera, 13-15, 28049 Madrid, Spain.}
    \email{\tt clara.torres@icmat.es}
    
    \begin{abstract}
        We obtain boundary nondegeneracy and regularity estimates for solutions to
        non-divergence form parabolic equations in parabolic $C^1$ domains, providing
        explicit moduli of continuity. Our results extend the classical Hopf-Oleinik
        lemma and boundary Lipschitz regularity for domains with
        $C^{1,\mathrm{Dini}}$ boundaries, while also recovering the known
        $C^{1-\varepsilon}$ regularity for parabolic Lipschitz domains, unifying both regimes
        with a single proof.
    \end{abstract}
    
    
    \maketitle

    \section{Introduction}\label{sect:intro}

    The Hopf-Oleinik lemma and boundary Lipschitz regularity for elliptic and parabolic equations are well understood in $C^{1,\mathrm{Dini}}$ domains, while in Lipschitz domains solutions are only H\"older continuous up to the boundary. In this paper, we prove boundary nondegeneracy and regularity estimates for solutions to parabolic non-divergence form equations
    $$\partial_t u - \mathcal{L}u = f$$
    in parabolic $C^1$ domains, providing explicit moduli of continuity. Our results recover the H\"older regularity in Lipschitz domains and the Hopf-Oleinik lemma and Lipschitz regularity in $C^{1,\mathrm{Dini}}$ domains, unifying both regimes with a single proof. To our knowledge, these estimates are new even for the heat equation.

    This is the second part of \cite{Tor26}, where the elliptic counterparts of Theorems \ref{thm:hopf} and \ref{thm:upper} were proved using a new approach based on barriers constructed from a locally regularized distance. The key feature of this construction is that the barrier at each point depends only on the boundary at a comparable scale, which allows meaningful estimates even when the $C^1$ modulus of continuity is not uniform. In the present paper, we adapt this construction to the parabolic setting.
    
    As a byproduct, we also obtain an endpoint interior regularity estimate (Theorem \ref{thm:interior_reg_modulus}), a parabolic counterpart of results by Teixeira \cite{Tei14} and Daskalopoulos, Kuusi, and Mingione \cite{DKM14}, which appears to be new.

    \newpage
    
    \subsection{Main results}
		Our first main result is a {\it nondegeneracy property} that becomes the Hopf-Oleinik lemma when the $C^1$ modulus of continuity of the domain is Dini.
	\begin{thm}\label{thm:hopf}
		Let $\L$ be as in \eqref{eq:non-divergence_operator}, let $\Omega$ satisfy the interior $C^1$ condition at $0$ with modulus $\omega$ in the sense of Definition \ref{defn:interiorC1}, and let $u$ be a nonnegative solution to 
		\[
		\p_t u - \L u = 0\quad \text{in} \quad \Omega.
		\]
		Then, for every $0 < \rho < \frac{r}{4} < \frac{r_0}{4}$,
		$$\frac{u(\rho e_n, 0)}{\rho} \geq \frac{1}{C}\frac{u(r e_n, -3r^2)}{r}\exp \left(-C\int_\rho^{2r}\omega(s)\frac{\mathrm{d}s}{s}\right),$$
		where $C$ and $r_0$ are positive and depend only on $\omega$, the dimension, and ellipticity constants.
	\end{thm}
	
	On the other hand, having an exterior $C^1$ condition ensures a bound on the growth of $\L$-harmonic functions at the boundary, which becomes linear when the $C^1$ modulus is Dini. The regularity up to the boundary follows by standard techniques.
	\begin{thm}\label{thm:upper}
		Let $\L$ be as in \eqref{eq:non-divergence_operator}, let $\Omega$ satisfy the exterior $C^1$ condition at $0$ with modulus $\omega$ in the sense of Definition \ref{defn:exteriorC1}, and let $u$ be a solution to
		$$\left\{\begin{array}{rclll}
			\p_t u - \L u & = & f & \text{in} & \Omega\cap Q_1\\
			u & = & g & \text{on} & \p_p\Omega\cap Q_1,
		\end{array}\right.$$
		where $f \in L^{n+1}$, $\omega_f(s) := \sup\limits_{x \in \Omega\cap Q_1}s^{-1/(n+1)}\|f\|_{L^{n+1}(Q_s(x))}$, and $g \in C^{1,\omega_g}_p$. 
		
		Then, for every $0 < \rho < \frac{r}{4} < \frac{r_0}{4}$,
		$$\frac{\|v\|_{L^\infty(Q_\rho)}}{\rho} \leq C\left(\frac{\|v\|_{L^\infty(Q_r)}}{r} + \int_\rho^{2r}[\omega_f+\omega_g](s)\frac{\mathrm{d}s}{s}\right)\exp \left(C\int_\rho^{2r}\omega(s)\frac{\mathrm{d}s}{s}\right),
        $$
		where $v = u - g(0, 0) - \nabla\!_x g (0, 0)\cdot x'$.
		
		Moreover, if $\L$ has doubly Dini continuous coefficients, $u \in C^{\tilde\omega}_p(\overline{\Omega}\cap Q_{1/2})$, where
		$$\tilde\omega({\theta}) = C{\theta}\left(\|u\|_{L^\infty(Q_1)}+\int_{\theta}^{2r_0}[\omega_f+\omega_g](s)\frac{\mathrm{d}s}{s}\right)\exp \left(C\int_{\theta}^{2r_0}\omega(s)\frac{\mathrm{d}s}{s}\right)$$
		is a modulus of continuity. The constants $C$ and $r_0$ are positive and depend only on $\omega$, $\omega_f$, $\omega_g$, the dimension, and ellipticity constants; in the second part, they also depend on the modulus of continuity of the coefficients of $\L$.
	\end{thm}
	
	\subsection{Background}

    For the elliptic background on the Hopf-Oleinik lemma and boundary regularity, we refer to \cite{Tor26} and the survey \cite{AN22}. Here we focus on the parabolic literature.
    
    The boundary point lemma was extended to the parabolic setting by Friedman \cite{Fri58} and V\'yborn\'y \cite{Vyb57}, building on Nirenberg's parabolic strong maximum principle \cite{Nir53}. Their results required domains with the interior ball condition. Kamynin and Khimchenko \cite{KK73} proved the parabolic Hopf-Oleinik lemma in parabolic $C^{1,\mathrm{Dini}}$ domains for classical solutions, and Kamynin \cite{Kam88} later showed that under an interior cone condition, the normal derivative inequality holds at some point in every neighbourhood of $x_0$, though not necessarily at $x_0$ itself. On the other hand, Lieberman \cite{Lie85} constructed a parabolic regularized distance adapted to parabolic scaling, and used it to derive the boundary point lemma via barriers under a Dini condition. More recently, Nazarov \cite{Naz12} obtained sharp versions of the Hopf-Oleinik lemma for parabolic non-divergence equations with unbounded lower-order coefficients, using the iterative method of Ladyzhenskaya and Ural'tseva \cite{LU88}.

    Regarding boundary regularity, Kamynin and Khimchenko \cite{KK74,KK75} established boundary Lipschitz estimates for parabolic equations in $C^{1,\mathrm{Dini}}$ domains. Boundary gradient estimates were also obtained by Ladyzhenskaya and Ural'tseva \cite{LU88}, and in their sharpest form by Nazarov \cite{Naz12}, who identified optimal conditions on the lower-order coefficients. More recently, Ma, Moreira, and Wang \cite{MMW17} proved boundary differentiability with explicit moduli of continuity for fully nonlinear parabolic equations under a Reifenberg $C^{1,\mathrm{Dini}}$ condition, and Dong, Li, and Lian \cite{DLL25} obtained boundary Lipschitz estimates and a Hopf lemma for Pucci-class parabolic equations in $C^{1,\mathrm{Dini}}$ domains. However, all of these results require the $C^{1,\mathrm{Dini}}$ condition, leaving open the question of what happens for rougher $C^1$ boundaries, a situation that arises naturally in free boundary problems such as the Stefan problem \cite{HR19,FRT26}.

    \subsection{Plan of the paper}
    This paper is organized as follows.

    We begin in Section \ref{sect:previ} by introducing the setting of the problem. Section \ref{sect:aux} collects auxiliary results used throughout the manuscript, such as an almost positivity property and the boundary Harnack inequality. Then, in Section \ref{sect:cones} we build barriers based on a localized regularized distance. Finally, Section \ref{sect:C1} is devoted to the proofs of our main results: Theorems \ref{thm:hopf} and~\ref{thm:upper}. The proof of the interior regularity estimate, Theorem \ref{thm:interior_reg_modulus}, is given in Appendix \ref{sect:app}.
	
    \subsection*{Acknowledgements}
    
    Much of this work was carried out during the visit of C.T.-L. to Salzburg. The authors wish to thank Prof. Verena Bögelein and the Department of Mathematics at the Paris Lodron Universität Salzburg for their support and hospitality.
    
    C.T.-L. has received funding from the European Research Council (ERC) under the Grant Agreement No. 862342, from AEI project PID2024-156429NB-I00 (Spain), and from the Grant CEX2023-001347-S funded by
    
    \noindent MICIU/AEI/10.13039/501100011033 (Spain).
    
    P.D.S.A. is partially supported by the Austrian Science Fund (FWF)
    
    \noindent [10.55776/P36295].

    \newpage
    
    \section{Preliminaries}\label{sect:previ}
	
	\subsection{Setting}
	In this section, we introduce the notation and basic definitions that will be used throughout the paper. Given $x \in \R^n$, we write $x' = (x_1,\ldots,x_{n-1})$. $B_r(x)$ denotes the open ball of radius $r$ of $\R^n$, centered at $x$, and $B'_r(x')$ stands for the open ball of $\R^{n-1}$, with center at $x'$. We also introduce the {\it parabolic cylinders}
	\[Q_r(x,t) := B_r'(x')\times(x_n-r,x_n+r)\times(t-r^2,t) \subset \R^{n+1}
    \]
    and 
	\[
   Q_r'(x,t) := B_r'(x')\times(t-r^2,t) \subset \R^n.
	\]
	When $(x, t) =(0, 0)$, we simply write $Q_r := Q_r(0,0)$ and $Q_r' := Q_r'(0,0)$.
	
    Let $\Omega \subset \R^{n}\times \R$ be an open connected set. For points $(x, t), (y, s) \in \Omega$, we define the {\it parabolic distance} by
	\[
	{\rm d_p}((x, t),(y, s)) := |x-y|+|t-s|^{1/2}.
	\]
	  Using this distance, we define the {\it parabolic Lipschitz seminorm $(\alpha = 1)$ and parabolic H\"older seminorm $(\alpha \in (0,1))$} of a function $g : \Omega \to \R$ as
	$$
	[g]_{C^{0,\alpha}_p(\Omega)} =\sup_{ \substack{(x,t),(y,s)\in \Omega \\ (x,t) \neq (y,s) }} \frac{|g(x,t)-g(y,s)|}{(|x-y|+|t-s|^{1/2})^\alpha}.
	$$
	This means that $g$ is $\alpha$-H\"older continuous with respect to the spatial variables and $\frac{\alpha}{2}$-H\"older continuous with respect to the time variable. 
    
    We define the {\it parabolic Lipschitz and Hölder norms}
	$$
	\|g\|_{C^{0,\alpha}_p(\Omega)} := \|g\|_{L^\infty(\Omega)} + [g]_{C^{0,\alpha}_p(\Omega)}.
	$$
    
    We use standard notation for Sobolev spaces: $W^{k,p}(\Omega)$ denotes the space of functions whose weak derivatives up to order $k$ belong to $L^p(\Omega)$. We write $W^{k,p}_{x,\,\mathrm{loc}}(\Omega)$ to indicate that only spatial derivatives are considered and the property holds locally, and similarly $W^{k,p}_{t,\,\mathrm{loc}}(\Omega)$ for time derivatives. 
    
    Throughout the text, we omit the domain from the notation whenever it does not lead to confusion, both for Hölder and Sobolev spaces.
    
    Next, we introduce the notion of {\it parabolic Lipschitz domains}.
	\begin{defn}\label{defn:lip_domain}
		We say that $\Omega$ is a {\it parabolic Lipschitz domain} in $Q_R$ with Lipschitz constant $L$ if $\Omega$ is a space-time domain whose boundary can be locally represented as the epigraph of a parabolic Lipschitz function 
        \[
        {\Gamma : B_R'\times[-R^2,0] \to \R}, \,  \, \, \Gamma(0,0) = 0,
        \]
        such that 
		\[
        \Omega = \big\{(x,t) \in Q_R \ | \ x_n > \Gamma(x',t) \big\}, \quad \|\Gamma\|_{C_p^{0,1}} \leq L.
        \]
		In this context, we denote the {\it lateral boundary}
		$$\p_\Gamma\Omega := \big\{(x,t) \in Q_R \ | \ x_n = \Gamma(x',t) \big\},$$
		and the {\it parabolic boundary}
		$$\p_p\Omega := \p_\Gamma\Omega\cup\big(\overline{\Omega}\cap\p Q_R\cap\{t < 0\}\big).$$
	\end{defn}

For simplicity, we restrict our attention to the following class of moduli of continuity.

	\begin{defn}\label{defn:modulus}
		We say that a function $\omega : [0,{\eta}_0) \to [0,\infty)$ is a \emph{modulus of continuity} if it satisfies the following properties:
    \begin{enumerate}
    \item $\omega$ is continuous on $[0, \eta_0)$,
    \item $\omega$ is strictly increasing,
    \item $\omega(0) = 0$.
\end{enumerate}
  	\end{defn}  

Now, we define regularity spaces with respect to a general modulus of continuity.
\begin{defn}
Let $\omega$ be a modulus of continuity. We say that $u$ belongs to $C^{0,\omega}_p$ if
$$
\|u - u(x_0,t_0)\|_{L^\infty(Q_r(x_0,t_0))} \leq \omega(r)
$$
for all parabolic cylinders $Q_r(x_0,t_0)$ contained in the domain of $u$.

We say that $u$ belongs to $C^{1,\omega}_p$ if its spatial gradient $\nabla\!_x u$ exists and satisfies
$$
\|\nabla\!_x u - \nabla\!_x u(x_0,t_0)\|_{L^\infty(Q_r(x_0,t_0))} \leq \omega(r),
$$
and additionally
$$
\|u - u(x_0,t_0) - \nabla\!_xu(x_0,t_0)(\cdot-x_0)\|_{L^\infty(Q_r(x_0,t_0))} \leq r\omega(r)
$$
for all parabolic cylinders $Q_r(x_0,t_0)$ contained in the domain of $u$.
\end{defn}

   Next, we introduce the interior pointwise $C^1$ condition for parabolic Lipschitz domains.
   
	\begin{defn}\label{defn:interiorC1}
		Let $\Omega$ be a parabolic Lipschitz domain. We say that $\Omega$ satisfies the {\it interior $C^1$ condition} at $(x_0, t_0) \in \p_p\Omega$ with modulus of continuity $\omega$ if, after a spatial rotation and a translation that send $(x_0, t_0)$ to the origin, there exists $\eta_0>0$ such that
		$$
		\left\{x_n > \left(|x'| + |t|^{\frac{1}{2}}\right)\omega(|x'| + |t|^{\frac{1}{2}})\right\}\cap Q_{{\eta}_0} \subset \Omega.
        $$
\end{defn}
Similarly, we introduce the exterior pointwise $C^1$ condition for parabolic Lipschitz domains.
	\begin{defn}\label{defn:exteriorC1}
		Let $\Omega$ be a parabolic Lipschitz domain. We say that $\Omega$ satisfies the {\it exterior $C^1$ condition} at  $(x_0, t_0) \in \p_p\Omega$ with modulus of continuity $\omega$ if, after a spatial rotation and a translation that sends $(x_0, t_0)$ to the origin, there exists $\eta_0>0$ such that
		\[
        \left\{x_n < -\left(|x'| + |t|^{\frac{1}{2}}\right)\omega \left(|x'| + |t|^{\frac{1}{2}}\right)\right\}\cap  Q_{{\eta}_0} \cap\Omega = \emptyset.
        \]
	\end{defn}

	The following class of moduli of continuity plays a distinguished role in regularity theory.
	\begin{defn}\label{defn:Dini}
		Let $\omega$ be a modulus of continuity. We say that $\omega$ is {\it Dini continuous} if the integral
		$$
        \int_0^{\eta}\omega(s)\frac{\mathrm{d}{s}}{s} < \infty \, \, \, \text{for some} \, \, \, \eta>0.
        $$
	\end{defn}
	
   \begin{defn}
Let $\omega$ be a modulus of continuity. We say that $\omega$ satisfies the \emph{double Dini condition} if the integral
\[
\int_0^{\eta} \frac{\mathrm{d}\xi}{\xi} \int_0^{\xi} \omega(s)\frac{\mathrm{d}s}{s} < \infty \, \, \, \text{for some} \, \, \, \eta>0.
\]
\end{defn}

    We consider non-divergence form second-order operators with bounded measurable coefficients.
\begin{defn} \label{sec_operator}
We say that $\mathcal{L}$ is a non-divergence form elliptic operator if
\begin{equation}\label{eq:non-divergence_operator}
	\mathcal{L}u = \sum\limits_{i,j = 1}^na_{ij}(x, t)\p^2_{ij}u,
\end{equation}
where $A = (a_{ij})_{i,j=1}^n$ is a symmetric matrix with bounded measurable coefficients satisfying $\lambda I \leq A(x,t) \leq \Lambda I$ for some $0 < \lambda \leq \Lambda$. The constants $\lambda$ and $\Lambda$ are referred to as the {\it ellipticity constants}.
\end{defn}

The \emph{Pucci extremal operators} \cite{Puc66} are the following:
	$$
	\mathcal{M^-}(D^2u) := \inf\limits_{\lambda I \leq A \leq \Lambda I}\operatorname{Tr}(AD^2u), \quad \mathcal{M^+}(D^2u) := \sup\limits_{\lambda I \leq A \leq \Lambda I}\operatorname{Tr}(AD^2u).
	$$
Note that the Pucci operators depend on the same ellipticity constants $\lambda, \Lambda$ as the operator $\mathcal{L}$, and it holds that $\mathcal{M^-}u\leq \L u \leq \mathcal{M^+}u$. For further details and properties of the Pucci extremal operators, we refer the reader to \cite{CC95} and \cite{FR22}. 
    
    In what follows, we introduce the notion $L^{n+1}$-viscosity solutions, which will be used throughout the manuscript.

	\begin{defn}[\cite{CKS00}]\label{defn:Ln-viscosity}
		Let $u \in C(\Omega)$, $f \in L^{n+1}_{\mathrm{loc}}(\Omega)$, and $\mathcal{L}$ as in \eqref{eq:non-divergence_operator}. We say that $u$ is a {\it $L^{n+1}$-viscosity subsolution} (resp. {\it supersolution}) to 
        \[
        \p_tu - \L u = f \, \, \,  \text{in} \, \, \, \Omega
        \] if, for all ${\varphi \in W^{2,n+1}_{x,\,\mathrm{loc}}(\Omega)}\cap W^{1,n+1}_{t,\,\mathrm{loc}}$ such that $u - \varphi$ has a local maximum (resp. minimum) at $(x_0,t_0)$,
		$$\operatorname{ess}\liminf\limits_{x \rightarrow x_0, \, t \rightarrow t_0} \p_t\varphi - \mathcal{L}\varphi - f \leq 0$$
		$$(\text{resp.} \ \operatorname{ess}\limsup\limits_{x \rightarrow x_0, \, t \rightarrow t_0} \p_t \varphi - \mathcal{L}\varphi - f \geq 0).
        $$
		A function $u$ is said to be an {\it $L^{n+1}$-viscosity solution} if it is both a subsolution and a supersolution.
	\end{defn}

\section{Auxiliary results}\label{sect:aux}
In what follows, we state an auxiliary lemma that quantifies how the vertical distance from a point to the boundary is comparable to its parabolic distance, with the proportionality controlled by the Lipschitz constant $L$.
 
\begin{lem}\label{lem:dp_vs_dh}
		Let $\Omega$ be a parabolic Lipschitz domain in $Q_1$ with Lipschitz constant $L$. Then, for any $(x,t) \in \Omega$,
		$$
        \operatorname{d_p}((x,t),\p_\Gamma\Omega) \leq x_n - \Gamma(x',t) \leq \left(\sqrt{1+L^2} \right)\operatorname{d_p}((x,t),\p_\Gamma\Omega).
        $$
	\end{lem}
	
	\begin{proof}
		The first inequality comes from the fact that $\operatorname{d_p}$ is a distance and the point $(x',\Gamma(x',t),t)$ belongs to $\p_\Gamma\Omega$. For the second inequality, note that
		\begin{align*}
			\operatorname{d_p}((x,t),\p_\Gamma\Omega) &= \inf\limits_{r \geq 0} \left\{r + \operatorname{d_x}(x,t+r^2,\p_\Gamma\Omega),r+\operatorname{d_x}(x,t-r^2,\p_\Gamma\Omega)\right\}\\
			&\geq \inf_{r \geq 0}\left\{r + \frac{x_n - \Gamma(x',t+r^2)}{\sqrt{1+L^2}},r + \frac{x_n - \Gamma(x',t-r^2)}{\sqrt{1+L^2}}\right\}\\[0.05cm]
			&\geq \inf_{r \geq 0}\left\{r + \frac{x_n - \Gamma(x',t) - Lr}{\sqrt{1+L^2}}\right\} = \frac{x_n - \Gamma(x',t)}{\sqrt{1+L^2}},
		\end{align*}
		where 
		$$\operatorname{d_x}((x,t),\p_\Gamma\Omega) := \operatorname{dist}\left(x,\{(y',\Gamma(y',t)\}\right)$$
		denotes the Euclidean distance in the $\mbox{x}$ direction at a fixed time.
	\end{proof}
    
	The following property of moduli of continuity will be used repeatedly.
	\begin{lem}[\protect{\cite[Lemma 2.10]{Tor26}}]\label{lem:exp_modulus_monotonicity}
		Let $\omega_1,\omega_2 : [0,{\theta}_0) \to [0,\infty)$ be moduli of continuity. Then, for every $a,c > 0$ and $b \in (0,{\theta}_0)$, there is $\tilde {\theta}_0 > 0$ such that $\tilde\omega : [0,\tilde {\theta}_0) \to [0,\infty)$ given by
		\[
        \tilde\omega({\theta}) := {\theta}\left(a+\int_{\theta}^b\omega_1(s)\frac{\mathrm{d}s}{s}\right)\exp \left(c\int_{\theta}^b\omega_2(s)\frac{\mathrm{d}s}{s} \right)
        \]
		is also a modulus of continuity.
	\end{lem}
    
    In what follows, we introduce the interior Harnack inequality, which will be an important tool in proving our main results.
	\begin{thm}[\protect{\cite[Lemma 5.2]{CKS00}}]\label{thm:Harnack}
		Let $\L$ be as in \eqref{eq:non-divergence_operator}, and let $u$ be a nonnegative solution to $\p_t u - \L u = 0$ in $Q_1$. Then, 
		$$\sup\limits_{Q_{1/2}(0, -\frac{1}{2})} u \leq C\inf\limits_{Q_{1/2}} u,$$
		where the constant $C>0$ depends only on the dimension and ellipticity constants.
	\end{thm}
	
	Then, we recall the Alexandrov-Bakelman-Pucci-Krylov-Tso estimate, originally proved in \cite{Kry76,Tso85}.
	\begin{thm}[\protect{\cite[Theorem 2.1]{CKS00}}]\label{thm:ABPKT}
		Let $\L$ be a non-divergence form operator as in \eqref{eq:non-divergence_operator} and let $u$ be a solution to $\p_t u - \L u = f$ in $Q_r$, with $f \in L^{n+1}(Q_r)$. Then,
		$$\sup\limits_{Q_r} u \leq \sup\limits_{\p_p Q_r} u^+ + Cr^{n/(n+1)}\|f\|_{L^{n+1}(Q_r)},$$
		where $C$ depends only on the dimension and the ellipticity constants.
	\end{thm}

 The following interior gradient estimate will play an important role in the proofs of our results.
    \begin{thm}[\protect{\cite[Theorem 7.3 \& Remark 7.7]{CKS00}}]\label{thm:gradient_estimate}
        Let $\L$ be a non-divergence form operator as in \eqref{eq:non-divergence_operator}, and assume that its coefficients $a_{ij}(x,t)$ are continuous. Let $u$ be a solution to $\p_tu - \L u = 0$ in $Q_1$. Then,
        $$\|u\|_{C^{0,1}_p(Q_{1/2})} \leq C\|u\|_{L^\infty(\p_pQ_1)},$$
        where $C$ depends only on the dimension, ellipticity constants, and the modulus of continuity of the coefficients of $\L$.
    \end{thm}

    Furthermore, for equations with a right-hand side in $L^{n+1}$, we will use the following endpoint regularity estimate, which is the parabolic counterpart of results by Teixeira \cite{Tei14} and Daskalopoulos-Kuusi-Mingione \cite{DKM14} and, to the best of our knowledge, does not appear in the literature.
    \begin{thm}\label{thm:interior_reg_modulus}
        Let $\L$ be a non-divergence form operator as in \eqref{eq:non-divergence_operator}, and assume that its coefficients are uniformly doubly Dini continuous. Let $u$ be a solution to
        $$\left\{
        \begin{array}{rclll}
        \p_t u - \L u & = & f & \text{in} & Q_1\\
        u & = & 0 & \text{on} & \p_p Q_1,
        \end{array}\right.$$
        where $f \in L^{n+1}$. Then, for every $(x,t), (y,s) \in Q_{1/4}$,
        $$|u(x,t) - u(y,s)| \leq Cd\int_d^4\omega_f(r)\frac{\mathrm{d}r}{r},$$
        where $d = \operatorname{d_p}((x,t),(y,s))$, $\omega_f(r) := \sup\limits_{x \in Q_1}r^{-1/(n+1)}\|f\|_{L^{n+1}(Q_r(x))}$, and $C$ is a positive constant that depends only on the dimension, ellipticity constants, and the modulus of continuity of the coefficients of $\L$.
    \end{thm}

    We defer the proof of Theorem \ref{thm:interior_reg_modulus} to Appendix \ref{sect:app}.
    
	The following {\it almost positivity property} can be understood as a quantitative form of the maximum principle.	
	\begin{lem}\label{lem:almost_positivity}
		Let $\L$ be as in \eqref{eq:non-divergence_operator}, and let $\Omega$ be a parabolic Lipschitz domain in the sense of Definition \ref{defn:lip_domain} with Lipschitz constant $L \leq \frac{1}{16}$ and $0 \in \p_\Gamma\Omega$. Let $u$ satisfy
		\[
		\left\{\begin{array}{rclll}
			\p_t u - \L u & = & 0 & \text{in} & \Omega\cap Q_1\\
			u & \geq & \!-\mu_0 & \text{in} & \Omega\cap Q_1\\
			u & = & 0 & \text{on} & \p_\Gamma\Omega\cap Q_1\\
			u\left(\frac{e_n}{2}, -\frac{3}{4}\right)\hspace{-0.3em} & \geq & 1,
		\end{array}\right.
        \]
		 where the constant $\mu_0 > 0$ depends only on the dimension and ellipticity constants. Then, 
        \[
        u \geq 0 \quad \text{in} \ \Omega\cap Q_{1/2}.
		\]
       
	\end{lem}
	
	\begin{proof}
		We will first prove that $u \geq 0$ in the segment $\{s e_n,\, s \in [0,\frac{1}{2}]\}.$
		
		Let $\delta \in (0,\frac{1}{8})$ to be chosen later and $(x, t) \in Q_{1/2} \cap \Omega_{\delta}$ an arbitrary point, where $\Omega_\delta := \{(x, t) \in Q_{1-\delta} : \operatorname{d_p}((x, t),\p_p\Omega) > \delta\}$.
		
		Now consider the segment $S$ joining $(\frac{e_n}{2},-\frac{3}{4})$ and $(x,t)$. We represent an arbitrary point in $S$ by $(y,s) = \left(\mu x + (1-\mu)\frac{e_n}{2},\mu t-(1-\mu)\frac{3}{4}\right)$, with $\mu \in [0,1]$. Then, we compute
		\begin{align*}
			y_n - \Gamma(y',s) &= \mu x_n + \frac{1-\mu}{2} - \Gamma\left(\mu x',\mu t - \frac{3(1-\mu)}{4}\right)\\
			&\geq \mu x_n + \frac{1-\mu}{2} - \mu\Gamma(x',t) - (1-\mu)\Gamma\left(0,-\frac{3}{4}\right)\\
			&\qquad - 2\mu(1-\mu)L\left(|x'|+\left|\frac{3}{4}+t\right|^{1/2}\right)\\
			&\geq \mu (x_n - \Gamma(x',t)) + (1-\mu)\left(\frac{1}{2} - \Gamma\left(0,-\frac{3}{4}\right)\right) - 4\mu(1-\mu)L\\
			&\geq \mu\delta + (1-\mu)\left(\frac{1}{2}-L\right) - 4\mu(1-\mu)L\\
			&= \mu\delta + (1-\mu)\left(\frac{1}{2} - L - 4\mu L\right) \geq \mu\delta + \frac{1-\mu}{8} > \delta.
		\end{align*}
		By Lemma \ref{lem:dp_vs_dh}, $\operatorname{d_p}(S,\p_\Gamma\Omega) > \delta/2$. Since $S \subset Q_{1-\delta}\cap\{t \geq -\frac{3}{4}\}$, also $\operatorname{d_p}(S,\p_p\Omega) > \delta/2$.
		
		Let now
		$$k := \left \lceil \frac{32}{\delta^2} \right \rceil,$$
		and consider a uniform partition $x_0 = x, x_1, \ldots, x_k = \frac{e_n}{2}; t_0 = t, t_1, \ldots, t_k = -\frac{3}{4}$. Let
		$$r := \sqrt{2\cdot\frac{\frac{3}{4}-|t|}{k}} \quad \Rightarrow \quad \frac{1}{\sqrt{k}} \leq r \leq \sqrt{\frac{2}{k}} \leq \frac{\delta}{4}.$$
		Then,
		$$|x_i - x_{i+1}| \leq \frac{1}{k} < \frac{r}{2} \quad \text{and} \quad t_i - t_{i+1} = \frac{r^2}{2}.$$
		Moreover, since $\operatorname{d_p}((x_i,t_i),\p_p\Omega) > \delta/2 \geq 2r$, $Q_{r}(x_i,t_i) \subset \Omega$.
		
		Let $v := u + \mu_0$, and note that it satisfies 
		$$\left\{\begin{array}{rclll}
			\p_t v - \L v & = & 0 & \text{in} & \Omega\cap Q_1\\
			v & \geq & 0 & \text{in} & \Omega.
		\end{array}\right.$$
		By the Harnack inequality, Theorem \ref{thm:Harnack}, we obtain
		$v(x_i,t_i) \geq cv(x_{i+1},t_{i+1})$ for some $c > 0$, and therefore 
		$$
		v(x,t) \geq c^kv\left(\frac{e_n}{2},-\frac{3}{4}\right) \geq c^k(1+\mu_0),$$
		using that $u(\frac{e_n}{2},-\frac{3}{4}) \geq 1$. Moreover, choosing $\mu_0$ smaller than $c^k/2$, we get that $u \geq c^k/2$ in $Q_{1/2}\cap\Omega_{\delta/2}$.
		
		Now, we define  $\tilde u = 2c^{-k} u$. Then $\tilde{u}$ solves $\p_t \tilde{u} - \L  \tilde{u}=0$, vanishes on $\p_p\Omega$, and satisfies
		$$\left\{\begin{array}{rclll}
			\tilde u & \geq & 1 & \text{in} & Q_{1/2} \cap \Omega_{\delta/2}\\
			\tilde u & \geq & -2c^{-k}\mu_0 & \text{in} & \Omega.
		\end{array}\right.$$
		Then, choosing $\delta$ and $\mu_0$ adequately small, by \cite[Lemma 2.5]{DS22} we conclude that $u \geq 0$ in the segment $\{se_n,\, s \in [0,\frac{1}{2}]\}$.
		
		Now, let $(x_0,t_0) \in \p_\Gamma\Omega\cap Q_{1/2}$ be an arbitrary point. By an analogous reasoning (choosing $\mu_0$ uniformly smaller if needed), $u \geq 0$ in $\{(x_0 + se_n,t_0),\, s \in [0,\frac{1}{2}]\}$, and we conclude that $u \geq 0$ in $\Omega\cap Q_{1/2}$.
	\end{proof}
	
	Finally, let us recall the parabolic boundary Harnack inequality. We refer the reader to \cite[Theorem 1.2]{Tor24} for a streamlined proof.
	\begin{thm}\label{thm:boundary_harnack}
		Let $\L$ be as in \eqref{eq:non-divergence_operator}, and let $\Omega$ be a parabolic Lipschitz domain in $Q_1$ in the sense of Definition \ref{defn:lip_domain} with Lipschitz constant $L \leq 1$ and $0 \in \p_p\Omega$. Let $m\in (0, 1]$, and let $u$ and $v$ be solutions\footnote{The original result in \cite{Tor24} is written for strong solutions, but when the right-hand side of the equation is zero, the proof also works for $L^{n+1}$-viscosity solutions, even in the case of bounded measurable coefficients.} to
		$$\left\{\begin{array}{rclll}
			\p_t u - \L u & = & 0 & \text{in} & \Omega\\
			u & = & 0 & \text{on} & \p_{\Gamma}\Omega.
		\end{array}\right.\quad\text{and}\quad\left\{\begin{array}{rclll}
			\p_t v - \L v & = & 0 & \text{in} & \Omega\\
			v & = & 0 & \text{on} &\p_{\Gamma}\Omega.
		\end{array}\right.$$
		such that $\| u\|_{L^{\infty} (Q_1)}\le 1$, $\| v\|_{L^{\infty} (Q_1)} = 1$, $v > 0$, and $v\left(\frac{e_n}{2}, -\frac{3}{4}\right) \ge m$. Then,
		$$
		\left\|\frac{u}{v}\right\|_{C^{0,\alpha}_p(\Omega \cap Q_{1/2})}  \leq C,
		$$
		where $C$ and $\alpha$ are a positive constants depending only on $m$, the dimension and ellipticity constants.
	\end{thm}
    
	\section{Barriers}\label{sect:cones}
	
	The goal of this section is to construct sub- and supersolutions of the form $d^{1\pm\varepsilon}$, where $d$ is a regularized distance, and $\varepsilon$ depends quantitatively on the oscillation of the boundary of the domain.
	
	We first introduce a \textit{locally regularized distance}. The advantage of this construction with respect to the more usual technique by Lieberman in \cite{Lie85} is that the behaviour of $d$ depends only on the boundary at a scale comparable to $d$, leading to improved pointwise bounds when the regularity of the boundary is not uniform.
	
	\begin{prop}\label{prop:regularized_distance}
		Let $\Omega$ be a parabolic Lipschitz domain in the sense of Definition \ref{defn:lip_domain} with Lipschitz constant $L \leq \frac{1}{C}$ and $0 \in \p_\Gamma\Omega$. Then, there exists a function $d : \Omega \cap Q_{1/2} \to \R$ satisfying the following:
		\begin{align*}
			1-C[\Gamma]_{C_p^{0, 1}(Q'_d(x',t))} &\leq \frac{d}{x_n - \Gamma(x', t)} \leq 1 + C[\Gamma]_{C_p^{0, 1}(Q'_d(x',t))},\\[0.1cm]
			1-C[\Gamma]_{C_p^{0, 1}(Q'_d(x',t))} &\leq |\nabla\!_x d| \leq 1 + C[\Gamma]_{C_p^{0, 1}(Q'_d(x',t))},\\[0.05cm]
			|\p_td|+|D^2d| &\leq \frac{C[\Gamma]_{C_p^{0, 1}(Q'_d(x',t))}}{d}.
		\end{align*}
		The constant $C>0$ depends only on the dimension.
	\end{prop}
	
	\begin{proof}
		We will construct $d$ as the inverse of a parametrization of $\Omega$.
        
		Let $\eta \in C^\infty_c(Q_1)$ be such that $\eta \geq 0$ and $\int\eta = 1$. Now, let
        $$P : \{(x',x_n,t) \in Q_1 : x_n > 0\} \to \Omega$$
        be defined as $P(x', x_n, t) = (x',p(x, t), t)$,	
		$$p(x, t) = (\eta_{x_n}*\Gamma)(x', t)+x_n := \int_{\R^{n-1}\times\R}x_n^{-(n+1)}\eta\left(\frac{z'}{x_n},\frac{s}{x_n^2}\right)\Gamma(x'+z',t+s)\mathrm{d}z'\mathrm{d}s+x_n.$$
		Then,
		\begin{align*}
			0 &= \p_n \left[\int_{\R^{n-1}\times\R}\eta_{x_n}\right]\\
			&= \int_{\R^{n-1}\times\R}-\frac{(n+1)}{x_n^{(n+2)}}\eta\left(\frac{z'}{x_n}, \frac{s}{x^2_n}\right)- \left(\frac{z'}{x_n^{(n+3)}} , \frac{2s}{x_n^{(n+4)}} \right)\cdot\nabla\eta\left(\frac{z'}{x_n}, \frac{s}{x^2_n}\right),
		\end{align*}
		and hence
		
		\begin{align*}
			\p_n p(x,t) &= \int_{\R^{n-1}\times\R}\p_n\left(x_n^{-(n+1)}\eta\left(\frac{z'}{x_n},\frac{s}{x_n^2}\right)\right)      \Gamma(x'+z', t+s)\mathrm{d}z'\mathrm{d}s + 1\\[0.1cm]
			&= \int_{\R^{n-1}\times \R}\p_n\left(x_n^{-(n+1)}\eta\left(\frac{z'}{x_n},\frac{s}{x_n^2}\right)\right) \left({\Gamma(x'+z', t+s)-\Gamma(x',t)}\right)\mathrm{d}z' \mathrm{d}s + 1\\[0.1cm]
			&\le \int_{\R^{n-1}\times \R}\left|\p_n\left(x_n^{-(n+1)}\eta\left(\frac{z'}{x_n},\frac{s}{x_n^2}\right)\right)\right| \left|{\Gamma(x'+z', t+s)-\Gamma(x',t)}\right|\mathrm{d}z' \mathrm{d}s + 1\\[0.1cm]
			&\le \int_{\R^{n-1}\times \R} \left|\p_n\left(x_n^{-(n+1)}\eta\left(\frac{z'}{x_n},\frac{s}{x_n^2}\right)\right)\right| [\Gamma]_{C_p^{0, 1}(Q'_{x_n}(x',t))}(|z'|+|s|^{1/2})\mathrm{d}z' \mathrm{d}s + 1\\[0.1cm]
			&\le C[\Gamma]_{C_p^{0, 1}(Q'_{x_n}(x',t))} + 1,
		\end{align*}
		where in the last inequality we used that
		\begin{align*}
			&\int_{\R^{n-1}\times \R} \left|\p_n\left(x_n^{-(n+1)}\eta\left(\frac{z'}{x_n},\frac{s}{x_n^2}\right)\right)\right| (|z'|+|s|^{1/2})\mathrm{d}z' \mathrm{d}s\\[0.1cm]
			&\qquad \lesssim \int_{\R^{n-1}\times \R} \left|x_n^{-(n+2)}\eta\left(\frac{z'}{x_n},\frac{s}{x_n^2}\right) - \left(\frac{z'}{x_n^{(n+2)}},\frac{2s}{x_n^{(n+3)}}\right)\cdot\nabla\eta\left(\frac{z'}{x_n},\frac{s}{x_n^2}\right)\right| x_n\mathrm{d}z' \mathrm{d}s\\[0.1cm]
			&\qquad \lesssim \int_{\R^{n-1}\times \R}x_n^{-(n+1)}\eta\left(\frac{z'}{x_n},\frac{s}{x_n^2}\right) + x_n^{-(n+1)}\left|\nabla\eta\left(\frac{z'}{x_n},\frac{s}{x_n^2}\right)\right|\mathrm{d}z' \mathrm{d}s \leq C.
		\end{align*}    
		
		Analogously, we can bound the integral below in the same way to obtain
		$$\p_n p(x,t) \geq 1 - C [\Gamma]_{C_p^{0, 1}(Q'_{x_n}(x',t))}.$$
		In summary,
		$$1 - C [\Gamma]_{C_p^{0, 1}(Q'_{x_n}(x',t))} \leq \p_n p(x, t) \leq 1 + C[\Gamma]_{C_p^{0, 1}(Q'_{x_n}(x',t))}.$$
		
		Hence, since $\operatorname{det}DP = \p_n p(x, t) > 0$, we can define $(y',d(y, t), t) = P^{-1}(y',y_n, t)$.
		
		Now, to estimate the derivatives of $d$, we first estimate the derivatives of $p$:
		\begin{itemize}

            \item For the time derivative, we first note that $\p_t$ acts on $\Gamma(x'+z', t+s)$, and since $\p_t\Gamma(x'+z',t+s) = \p_s\Gamma(x'+z',t+s)$, integrating by parts in $s$ gives
            \[
            \p_t p = -\int_{\R^{n-1} \times \R} x_n^{-(n+3)} \p_t\eta \left(\frac{z'}{x_n}, \frac{s}{x^2_n} \right)  \Gamma \left(x' + z' , t+s \right) \mathrm{d}z' \mathrm{d}s.
            \]
            Since $\int \p_t\eta = 0$, we can subtract $\Gamma(x',t)$ from $\Gamma(x'+z',t+s)$ inside the integral, and then
            \begin{align*}
            |\p_t p| & = \left| \displaystyle\int_{\R^{n-1} \times \R} x_n^{-(n+3)} \p_t\eta \left(\frac{z'}{x_n}, \frac{s}{x^2_n} \right)  \left(\Gamma \left(x' + z' , t+s \right) - \Gamma(x',t) \right)\mathrm{d}z' \mathrm{d}s \right|  \\[0.2cm]
            & \le \frac{C[\Gamma]_{C_p^{0, 1}(Q'_{x_n}(x',t))}}{x_n}.
            \end{align*}
            
			\item For $i = 1,\ldots,n-1$,
			
			$$\left|\p_i p \right| = \left|\int_{\R^{n-1} \times \R}x_n^{-(n+1)}\eta\left(\frac{z'}{x_n}, \frac{s}{x^2_n}\right) \p_i \Gamma(x'+z', t +s)\mathrm{d}z' \mathrm{d}s\right| \leq [\Gamma]_{C_p^{0, 1}(Q'_{x_n}(x',t))}.
			$$
			\item For $i,j = 1,\ldots,n-1$, integrating by parts,
			
			$$|\p^2_{ij}p| = \left|\int_{\R^{n-1} \times \R}x_n^{-(n+2)}\p_i\eta\left(\frac{z'}{x_n} , \frac{s}{x^2_n}\right)\p_j\Gamma(x'+z', t + s)\mathrm{d}z' \mathrm{d}s\right| \leq \frac{C[\Gamma]_{C_p^{0, 1}(Q'_{x_n}(x',t))}}{x_n}.$$

            \newpage
			\item For $i = 1,\ldots,n-1$,
			
			\begin{align*}
				|\p^2_{in}p| &= \left|\int_{\R^{n-1}\times \R}\left(-\frac{(n+1)}{x_n^{(n+2)}}\eta\left(\frac{z'}{x_n}, \frac{s}{x^2_n}\right) \right.\right.\\[0.1cm] 
				&\qquad\left.\left. -\, \left(\frac{z'}{x_n^{(n+3)}} , \frac{2s}{x_n^{(n+4)}} \right)\cdot\nabla\eta\left(\frac{z'}{x_n}, \frac{s}{x^2_n}\right)\right)
				\p_i\Gamma(x'+z', t+s)\mathrm{d}z' \mathrm{d}s\right|\\[0.1cm]
				&\leq \frac{C[\Gamma]_{C_p^{0, 1}(Q'_{x_n}(x',t))}}{x_n}.
			\end{align*}
			
			\item Finally,
			
			\begin{align*}
				|\p^2_{nn}p| &= \left|\int_{\R^{n-1}\times \R}\left(\frac{(n+1)(n+2)}{x_n^{(n+3)}}\eta\, \left(\frac{z'}{x_n},\frac{s}{x^2_n}\right)\right.\right.\\[0.1cm]
				&\left.\left.\qquad + \, \left( \frac{(2n+4)z'}{x_n^{(n+4)}}, \frac{(4n+10)s}{x_n^{(n+5)}}\right) \cdot\nabla\eta\left(\frac{z'}{x_n}, \frac{s}{x^2_n}\right)\right.\right.\\[0.1cm]
				&\left.\left.\qquad +\, \frac{1}{x_n^{(n+1)}}\left(\frac{z'}{x_n^{2}}, \frac{2s}{x_n^{3}}\right)^\top\cdot \! D^2\eta\left(\frac{z}{x_n}, \frac{s}{x^2_n}\right)\cdot \left(\frac{z'}{x_n^{2}},  \frac{2s}{x_n^{3}}\right)\right)\Gamma(x'+z', t+s)\mathrm{d}z'\mathrm{d}s\right|\\[0.1cm]
				&\leq \frac{C[\Gamma]_{C_p^{0, 1}(Q'_{x_n}(x',t))}}{x_n}.
			\end{align*}
		\end{itemize}
		
		The first derivatives of $d$ can be estimated as:
		\begin{itemize}
			\item For the time derivative
			\[
			|\p_t d| = \left|\frac{\p_tp}{\p_np}\right| \leq \frac{C[\Gamma]_{C_p^{0, 1}(Q'_{x_n}(x',t))}}{x_n}.
			\]
			\item For the last component,
			$$|\p_nd| = \frac{1}{|\p_np|} \leq 1+C[\Gamma]_{C_p^{0, 1}(Q'_{x_n}(x',t))}.$$
			\item For $i = 1,\ldots,n-1$, 
			$$|\p_id| = \left|\frac{\p_ip}{\p_np}\right| \leq C[\Gamma]_{C_p^{0, 1}(Q'_{x_n}(x',t))}.
			$$
		\end{itemize}
		
		To compute $D^2d$, we distinguish three cases:
		\begin{itemize}
			\item When $i, j = 1,\ldots,n-1$,
			$$|\p^2_{ij}d| \leq \frac{1}{|\p_np|}\left(|\p^2_{ij}p|+\frac{|\p^2_{in}p\hspace{0.1em}\p_jp|+|\p^2_{nj}p\hspace{0.1em}\p_ip|}{|\p_np|}+\frac{|\p^2_{nn}p\hspace{0.1em}\p_ip\hspace{0.1em}\p_jp|}{|\p_np|^2}\right) \leq \frac{C[\Gamma]_{C_p^{0, 1}(Q'_{x_n}(x',t))}}{x_n}.$$
			\item When $i = 1,\ldots,n-1$,
			$$|\p^2_{in}d| \leq \frac{1}{|\p_np|^2}\left(|\p_{in}p|+\frac{|\p^2_{nn}p\hspace{0.1em}\p_ip|}{|\p_np|}\right) \leq \frac{C[\Gamma]_{C_p^{0, 1}(Q'_{x_n}(x',t))}}{x_n}.$$
			\item Finally,
			$$|\p^2_{nn}d| = \frac{|\p^2_{nn}p|}{|\p_np|^3} \leq \frac{C[\Gamma]_{C_p^{0, 1}(Q'_{x_n}(x',t))}}{x_n}.$$
		\end{itemize}
	\end{proof}
	
	Using the previous estimates, we see that $d^{1+\varepsilon}$ is a subsolution and $d^{1-\varepsilon}$, a supersolution.
	\begin{lem}\label{lem:barriers}
		For any parabolic Lipschitz domain $\Omega$ with Lipschitz constant $L \leq \frac{1}{C_0}$, and for any $\varepsilon \geq C_0[\Gamma]_{C_p^{0, 1}(Q'_{2r}(x',t))},$
		$$
		\p_t d^{1+\varepsilon} - \mathcal M^-d^{1+\varepsilon} \leq 0 \quad \text{and} \quad \p_t d^{1-\varepsilon} - \mathcal M^+d^{1-\varepsilon} \geq 0 \quad \text{in} \ \Omega\cap Q_r,
		$$
		where a positive constant $C_0$ depends only on the dimension and ellipticity constants.
	\end{lem}
	
	\begin{proof}
		It follows from the definition of the Pucci operator combined with Proposition \ref{prop:regularized_distance},
		\begin{align*}
			\p_t d^{1+\varepsilon}  -  \mathcal M^-d^{1+\varepsilon} &= \p_t d^{1+\varepsilon} 
			- \inf\limits_{\lambda I \leq A \leq \Lambda I}\operatorname{Tr}(AD^2d^{1+\varepsilon})\\
			&= (1+\varepsilon)\left( d^{\varepsilon} \p_t d
			- \inf\limits_{\lambda I \leq A \leq \Lambda I}\left[\sum\limits_{i=1}^n\sum\limits_{j=1}^na_{ij}(d^\varepsilon\p^2_{ij}d + \varepsilon d^{-1+\varepsilon}\p_id\p_jd)\right]\right)\\
			&\leq (1+\varepsilon)\left( d^{\varepsilon} \p_t d
			- \left(d^\varepsilon\mathcal M^-d+\varepsilon d^{-1+\varepsilon}\inf\limits_{\lambda I \leq A \leq \Lambda I}\nabla\!_x d^\top A \nabla\!_x d\right) \right)\\
			&\leq (1+\varepsilon)\left(C\Lambda [\Gamma]_{C_p^{0, 1}(Q'_d(x',t))}-\frac{\lambda\varepsilon}{4}\right)d^{-1+\varepsilon} \leq 0.
		\end{align*}
		Note that in the last inequality, we use that $[\Gamma]_{C_p^{0, 1}(Q'_d(x',t))}$ is small. The inequality for $d^{1-\varepsilon}$ is analogous.
	\end{proof}
	
	Then, using these barriers we construct a solution with a controlled growth.
	\begin{prop}\label{prop:special_solns}
		Let $\L$ be as in \eqref{eq:non-divergence_operator}, and let $\Omega$ be a parabolic Lipschitz domain in the sense of Definition \ref{defn:lip_domain} with Lipschitz constant $L \leq L_0$. Then, for every $r \in (0,\frac{1}{3})$, there exists a solution $\varphi_r$ to
		$$\left\{\begin{array}{rclll}
			\p_t\varphi_r - \L\varphi_r & = & 0 & \text{in} & \Omega\cap Q_r\\
			\varphi_r & = & 0 & \text{on} & \p_\Gamma\Omega\cap Q_r
		\end{array}\right.$$
		with the growth estimates
		$$(2r)^{-\varepsilon}d^{1+\varepsilon} \leq \varphi_r \leq (2r)^\varepsilon d^{1-\varepsilon} \quad \text{in} \ \Omega\cap Q_r,$$
		and
		$$\|\varphi_r - d\|_{L^\infty(\Omega\cap Q_r)} \leq Kr[\Gamma]_{C_p^{0, 1}(Q'_{2r})},$$
		where $\varepsilon = C_0[\Gamma]_{C_p^{0, 1}(Q'_{2r})}$, and $C_0 >0$ is from Lemma \ref{lem:barriers}. The positive constants $L_0$ and $K$ depend only on the dimension and ellipticity constants.
	\end{prop}
	
	\begin{proof}
		First, by Proposition \ref{prop:regularized_distance}, $d \leq 2r$ in $\Omega\cap Q_r$, and then $(2r)^{-\varepsilon}d^{1+\varepsilon} \leq (2r)^\varepsilon d^{1-\varepsilon}$. Now, let $\varphi_r$ be the solution to
		$$\left\{\begin{array}{rclll}
			\p_t\varphi_r - \L\varphi_r & = & 0 & \text{in} & \Omega\cap Q_r\\
			\varphi_r & = & d & \text{on} & \p_p(\Omega\cap Q_r).
		\end{array}\right.$$
		From Lemma \ref{lem:barriers} and the comparison principle, it follows that
		$$(2r)^{-\varepsilon}d^{1+\varepsilon} \leq \varphi_r \leq (2r)^\varepsilon d^{1-\varepsilon} \quad \text{in} \ \Omega\cap Q_r.$$
		Moreover,
		\begin{align*}
			\|\varphi_r - d\|_{L^\infty(\Omega\cap Q_r)} &\leq \|(2r)^{-\varepsilon}d^{1+\varepsilon}-(2r)^\varepsilon d^{1-\varepsilon}\|_{L^\infty(\Omega\cap Q_r)}\\[0.05cm]
			&\leq \sup\limits_{\eta \in [0,2r]}(2r)^\varepsilon {\eta}^{1-\varepsilon} - (2r)^{-\varepsilon}{\eta}^{1+\varepsilon} \leq Cr\varepsilon.
		\end{align*}
	\end{proof}
	In what follows, we introduce an auxiliary lemma that will be useful for the proof of Theorem \ref{thm:hopf}. 
	
	\begin{lem}\label{Lemma_ATL}
		Let $\varphi_k := \varphi_{2^{-(k+1)}}$ and $\varepsilon_k := C_0[\Gamma]_{C_p^{0, 1}\left(Q'_{2^{-k}}\right)}$ be sequences defined as in Proposition \ref{prop:special_solns}. Define the functions $v$ and $w$ as follows:
		\begin{equation*}
			v(x, t) := \frac{\varphi_{k-1} \left(2^{-k-1}x, 2^{-2(k+1)}t\right)}{\varphi_{k-1} \left(2^{-k-2}e_n, -3 \cdot 2^{-2(k+2)}  \right)}\,  \, \,  \text{and} \, \, \,   w(x, t) := \frac{\varphi_k\left(2^{-k-1}x, 2^{-2(k+1)}t \right)}{\varphi_k\left(2^{-k-2}e_n, -3 \cdot 2^{-2(k+2)} \right)}.
		\end{equation*}
		Then, $v\left(\frac{e_n}{2}, -\frac{3}{4}\right) = w\left(\frac{e_n}{2}, -\frac{3}{4}\right) = 1$ and
		\[
		\|v - w\|_{L^\infty(\tilde\Omega\cap Q_1)} \le  M \varepsilon_{k-1},
		\]
		where $\tilde\Omega$ is the appropriate rescaling of $\Omega$, and $M$ is a positive constant depending only on the dimension and ellipticity constants.
	\end{lem}
	
	\begin{proof}
		The normalization condition is straightforward from the definition of $v$ and $w$. 
		
		Now we shall estimate the $L^{\infty}$ norm of $v - w$. For simplicity, we denote the point $ z_k:= (2^{-k-2}e_n,  -3 \cdot 2^{-2(k+2)})$. 
		
		First, we apply Proposition \ref{prop:special_solns} for $\varphi_{k-1}$ and $\varphi_k$ to obtain
		\begin{align*}
			\|\varphi_{k-1}-\varphi_k\|_{L^{\infty} (\Omega\cap Q_{2^{-k-1}})}
			&\le \|\varphi_{k-1}- d\|_{L^{\infty} (\Omega\cap Q_{2^{-k-1}})} + \| \varphi_k - d\|_{L^{\infty} (\Omega\cap Q_{2^{-k-1}})}\nonumber\\
			& \le 2^{-k}K[\Gamma]_{C_p^{0, 1}(Q'_{2^{-k+1}})} + 2^{-k-1} K [\Gamma]_{C_p^{0, 1}(Q'_{2^{-k}})} \nonumber\\
			& < 2^{-k+1}K[\Gamma]_{C_p^{0, 1}(Q'_{2^{-k+1}})},     
		\end{align*}
		\begin{equation*} 
			\| \varphi_{k -1} \|_{L^{\infty} (\Omega\cap Q_{2^{-k-1}})} \le 2^{-k}K [\Gamma]_{C_p^{0, 1}(Q'_{2^{-k +1}})} + 2^{-k},
		\end{equation*}
		and 
		\begin{equation*} 
			\| \varphi_{k} \|_{L^{\infty} (\Omega\cap Q_{2^{-k -1}})} \le 2^{-k -1}K [\Gamma]_{C_p^{0, 1}(Q'_{2^{-k}})} + 2^{-k}.
		\end{equation*}
		
		Hence,
		\begin{align*} \label{ineq:difference-estimates}
			\|v - w\|_{L^\infty(\tilde\Omega\cap Q_1)} 
			&= \left \|\frac{\varphi_{k-1}}{\varphi_{k-1}(z_k)} - \frac{\varphi_k}{\varphi_k(z_{k})}  \right\|_{\infty}   \\
			&\le  \left| \frac{1}{\varphi_{k-1}(z_k)} - \frac{1}{\varphi_k(z_{k})} \right| \|\varphi_{k-1}\|_{\infty} + \frac{1}{\varphi_{k}(z_k)} \left\| \varphi_{k-1} - \varphi_{k} \right\|_{\infty} \\
			& \le \frac{\left\| \varphi_{k-1} - \varphi_{k} \right\|_{\infty}}{\varphi_{k-1}(z_k) \varphi_k(z_{k})} \|\varphi_{k-1}\|_{\infty} + \frac{1}{\varphi_k(z_k)} \left\| \varphi_{k-1} - \varphi_{k} \right\|_{\infty}  \\
			&\leq \frac{2^{-k+1}K[\Gamma]_{C_p^{0, 1}\left(Q'_{2^{-k+1}}\right)}}{\left(2^{-k-3}-2^{-k}K[\Gamma]_{C_p^{0, 1}\left(Q'_{2^{-k+1}}\right)}\right)^2}\left(2^{-k} + 2^{-k}K[\Gamma]_{C_p^{0, 1}\left(Q'_{2^{-k+1}}\right)}\right)\, + \\
			&\quad+\frac{1}{2^{-k-3}-2^{-k}K[\Gamma]_{C_p^{0, 1}\left(Q'_{2^{-k}}\right)}}2^{-k+1}K[\Gamma]_{C_p^{0, 1}\left(Q'_{2^{-k+1}}\right)} \\
			& \le M \varepsilon_{k-1}.
		\end{align*}     
		where throughout the proof, we use the notation $\| \cdot \|_{\infty}$ to denote $\|\cdot\|_{L^\infty(Q_{2^{-k-1}})}$ for formatting.
	\end{proof}	
	
	In what follows, we introduce an analogous auxiliary result that we will use in the proof of Theorem \ref{thm:upper}.
	
	\begin{lem}\label{Lemma_ATL2}
		Let $\varphi_k := \varphi_{2^{-2k-1}}$ and $ \varepsilon_k := C_0[\Gamma]_{C_p^{0, 1}\left(Q'_{4^{-k}}\right)}$ be sequences defined as in Proposition \ref{prop:special_solns}. Let $v$ and $w$ be as follows:
		\begin{equation*}
			v(x, t) := \frac{\varphi_{k-1}\left(2^{-2k-1}x, 2^{-2(2k+1)}t\right)}{\varphi_{k-1}\left(2^{-2k-2}e_n, -3 \cdot 2^{-2(2k+2)}  \right)}, \ \text{and} \ w(x, t) := \frac{\varphi_k\left(2^{-2k-1}x, 2^{-2(2k+1)}t\right)}{\varphi_k\left(2^{-2k-2}e_n, -3 \cdot 2^{-2(2k+2)} \right)}.
		\end{equation*}
		Then, $v\left(\frac{e_n}{2}, -\frac{3}{4}\right) = w\left(\frac{e_n}{2}, -\frac{3}{4}\right) = 1$ and
		\[
		\|v - w\|_{L^\infty(\tilde\Omega\cap Q_1)} \le  M \varepsilon_{k-1},
		\]
		where $\tilde\Omega$ is the appropriate rescaling of $\Omega$, and $M$ is a positive constant depending only on the dimension and ellipticity constants.
	\end{lem}
	\begin{proof}
		The proof follows the same lines as the proof of Lemma \ref{Lemma_ATL}.
	\end{proof}

    \newpage
    
	\section{Quantitative nondegeneracy and regularity in \texorpdfstring{$C^1$}{C1} domains}\label{sect:C1}
	
	We first prove our boundary nondegeneracy estimate for domains with the interior $C^1$ condition.
	
	\begin{proof}[Proof of Theorem \ref{thm:hopf}]
		
		Let $u$ be a nonnegative solution to 
		\[
		\p_t u - \L u = 0\quad \text{in} \quad \Omega.
		\]
		
		Let $r_0 > 0$ small enough such that $\omega(r_0) < \omega_0$ to be determined later. Without loss of generality after rescaling and dividing by a constant, we may assume that $u\left(\frac{e_n}{2}, -\frac{3}{4}\right) = 1$ and $r_0 = 1$. In the sequel, we define the sequences 
		\[
		\varphi_k := \varphi_{2^{-k-1}} \quad \text{and} \quad \varepsilon_k := C_0[\Gamma]_{C_p^{0, 1}\left(Q'_{2^{-k}}\right)}
		\]
		as in Lemma \ref{Lemma_ATL}. From Proposition \ref{prop:special_solns}, we have that $\varphi_0$ solves
		$$
		\left\{\begin{array}{rclll}
			\p_t\varphi_0 - \L\varphi_0 & = & 0 & \text{in} & \Omega \cap Q_{1/2} \\
			\varphi_0& = & 0 & \text{on} & \p_\Gamma\Omega \cap Q_{1/2}.
		\end{array}\right.
		$$    
		We also introduce $u_0$ as the solution to
		$$
		\left\{\begin{array}{rclll}
			\p_tu_0 - \L u_0 & = & 0 & \text{in} & \Omega\\
			u_0 & = & 0 & \text{on} & \p_\Gamma\Omega\\
			u_0& = & u & \text{on} & (\p_p\Omega)\setminus\p_\Gamma\Omega.
		\end{array}\right.
		$$
		By the maximum principle, $u \geq u_0$.
		
		Now, by the boundary Harnack (Theorem \ref{thm:boundary_harnack}) applied to $u_0$ and $\varphi_0$, we deduce that
		$$u \geq u_0 \geq c_0\varphi_0\chi_{Q_{1/4}}.$$
		
		Now, we will see inductively that for $k \geq 1$,
		$$u \geq c_k\varphi_k\chi_{Q_{2^{-k-2}}}, \quad c_k := (1-A\varepsilon_{k-1})c_{k-1},$$
		where $A$ is a large constant to be chosen later depending only on the dimension and ellipticity constants. To prove that, we recall the functions $v$ and $w$ defined as in Lemma \ref{Lemma_ATL}, namely 
		$$
		v(x, t) := \frac{\varphi_{k-1}\left(2^{-k-1}x, 2^{-2(k+1)}t\right)}{\varphi_{k-1}\left(2^{-k-2}e_n, -3 \cdot 2^{-2(k+2)}  \right)}, \ \text{and} \ w(x, t) := \frac{\varphi_k\left(2^{-k-1}x, 2^{-2(k+1)}t\right)}{\varphi_k\left(2^{-k-2}e_n, -3 \cdot 2^{-2(k+2)} \right)}
		$$
		
		Next, we define the following auxiliary function 
		$$
		h := \frac{v - (1-\mu_0^{-1}M\varepsilon_{k-1})w}{\mu_0^{-1}M\varepsilon_{k-1}}.$$
		Denoting by $\tilde\L$ and $\tilde\Omega$ the appropriate rescalings of $\L$ and $\Omega$, $h$ is a solution to
		$$
		\left\{\begin{array}{rclll}
			\p_t h - \tilde\L h & = & 0 & \text{in} & \tilde\Omega\cap Q_1\\
			h & \geq & \!-\mu_0 & \text{in} & \tilde\Omega\cap Q_1\\
			h & = & 0 & \text{on} & \p_\Gamma\tilde\Omega\cap Q_1\\
			h\left(\frac{e_n}{2}, -\frac{3}{4}\right)\hspace{-0.3em} & = & 1, &&
		\end{array}\right.
		$$
        It is straightforward to check that $h$ solves the problem above; in particular, Lemma \ref{Lemma_ATL} gives $h \geq -\mu_0$. Then, by Lemma \ref{lem:almost_positivity}, $h \geq 0$ in $\tilde\Omega\cap Q_{1/2}$, that is,
        \[
        v \geq (1-\mu_0^{-1}M\varepsilon_{k-1})w\quad \text{in} \quad \tilde\Omega\cap Q_{1/2}.
        \]      
		
		Going back to our inductive argument, if we assume that $u \geq c_{k-1}\varphi_{k-1}$ in ${\Omega\cap Q_{2^{-k-1}}}$, this implies that
		\begin{align*}
			u(x, t) &\geq c_{k-1}\varphi_{k-1}\left(2^{-k-2}e_n, -3 \cdot 2^{-2(k+2)} \right)v \left(2^{k+1}x, 2^{2(k+1)}t \right) \\
			&\geq c_{k-1}\varphi_{k-1}\left(2^{-k-2}e_n,  -3 \cdot 2^{-2(k+2)} \right)(1-\mu_0^{-1}M\varepsilon_{k-1})w\left(2^{k+1}x, 2^{2(k+1)}t \right)\\
			&\geq c_{k-1}\frac{\varphi_{k-1}\left(2^{-k-2}e_n,  -3 \cdot 2^{-2(k+2)} \right)}{\varphi_{k}\left(2^{-k-2}e_n,  -3 \cdot 2^{-2(k+2)} \right)}(1-\mu_0^{-1}M\varepsilon_{k-1})\varphi_k\\
			&\geq c_{k-1}\frac{2^{-k-3}-2^{-k}K[\Gamma]_{C_p^{0, 1}\left(Q'_{2^{-k+1}}\right)}}{2^{-k-3}+2^{-k-1}K[\Gamma]_{C_p^{0, 1}\left(Q'_{2^{-k}}\right)}}(1-\mu_0^{-1}M\varepsilon_{k-1})\varphi_k \\
			&\geq c_{k-1}(1 - A\varepsilon_{k-1})\varphi_k,
		\end{align*}
		in $\Omega\cap Q_{2^{-k-2}}$, provided that $\varepsilon_{k-1}$ is small enough.
		
		Hence, by choosing $\omega_0$ small enough (so that $\varepsilon_j < 1/(2A)$ for all $j$), we compute
		$$c_k = c_0\prod\limits_{j = 0}^{k-1}(1 - A\varepsilon_j) \geq c_0\cdot 4^{-A\sum\limits_{j=0}^{k-1}\varepsilon_j}.$$
		Thus, for any $r \in [2^{-k-3},2^{-k-2}]$ combined with the growth condition for $\varphi_k$  in the Proposition \ref{prop:special_solns},
		\begin{align*}
			u(re_n, 0) &\geq c_k\varphi_k(re_n, 0) \\
			&\geq c_0\cdot 4^{-A\sum\limits_{j = 0}^{k-1}\varepsilon_j}\varphi_k(re_n, 0)\\
			&\geq c_0\cdot 4^{-A\sum\limits_{j = 0}^{k-1}\varepsilon_j}(2^{-k})^{-\varepsilon_k}\left(\frac{r}{2}\right)^{1+\varepsilon_k} \geq \frac{c_0}{32}\cdot4^{-A\sum\limits_{j = 0}^{k-1}\varepsilon_j}r.
		\end{align*}
		Therefore,
		$$\frac{u(re_n, 0)}{r} \geq \frac{1}{C}\exp \left(-C\sum\limits_{j=0}^{k-1}\omega(2^{-j}) \right)\geq \frac{1}{C}\exp \left(-C\int_{8r}^2\omega(s)\frac{\mathrm{d}s}{s}\right).$$
	\end{proof}
	
	The proof of the upper bound for domains with the exterior $C^1$ condition follows the same strategy, with some extra work to deal with the Dirichlet boundary condition and the source term.
	
	\begin{proof}[Proof of Theorem \ref{thm:upper}]
		Let $r_0 > 0$ small enough such that $\omega(r_0), \omega_g(r_0), \omega_f(r_0) < \omega_0$ to be determined later. Without loss of generality, after rescaling and subtracting a linear function, we may assume that $r_0 = 1$, $g(0, 0) = 0$ and $\nabla\!_x g(0, 0) = 0$. Then, we define the sequences $\varphi_k := \varphi_{2^{-2k-1}}, \varepsilon_k := C_0[\Gamma]_{C_p^{0, 1}\left(Q'_{4^{-k}}\right)}$ as introduced in Proposition~\ref{prop:special_solns}. Also, as a consequence of Proposition \ref{prop:special_solns}, we obtain that $\varphi_0$ is a solution to
		$$
		\left\{\begin{array}{rclll}
			\p_t\varphi_0 - \L\varphi_0 & = & 0 & \text{in} & \Omega \cap Q_{1/2}\\
			\varphi_0& = & 0 & \text{on} & \p_\Gamma\Omega \cap Q_{1/2}.
		\end{array}\right.
		$$
		
		Now, define $\tilde u_0$ as the solution to
		
		$$\left\{\begin{array}{rclll}
			\p_t \tilde u_0 -  \L \tilde u_0 & = & 0 & \text{in} & \Omega \cap Q_{1/2}\\
			\tilde u_0 & = & |u| & \text{on} & \p_p Q_{1/2} \cap \Omega \\
			\tilde u_0 & = & 0 & \text{on} & \p_\Gamma\Omega \cap Q_{1/2}.
		\end{array}\right.$$
		Applying the parabolic boundary Harnack (Theorem \ref{thm:boundary_harnack}) to $\tilde u_0$ and $\varphi_0$, we obtain
		$$\tilde u_0 \leq c_0\varphi_0 \quad \text{in} \ \Omega\cap Q_{1/4},
		$$
		where $c_0 \leq C\|u\|_{L^\infty(Q_{1/2})}$, and by the comparison principle and Theorem \ref{thm:ABPKT},
		\begin{equation} \label{induction_first-case}
			|u| \leq \tilde u_0 + \|g\|_{L^\infty(Q_{1/2})} + 2^{-n/(n+1)}C\|f\|_{L^{n+1}(\Omega\cap Q_{1/2})} \leq c_0\varphi_0 + d_0 \quad \text{in} \ \Omega\cap Q_{1/4},
		\end{equation}
		where $d_0 = \omega_g(1)+C\omega_f(1)$.
		
		We now prove, by induction, that for $k \geq 1$,
		$$|u| \leq c_k\varphi_k + 4^{-k}d_k \quad \text{in} \, \,  \Omega \cap Q_{4^{-k-1}},$$
		where 
		$$c_k := (1+A\varepsilon_{k-1})c_{k-1} + Ad_{k-1} \quad \text{and} \quad d_k := \omega_g( 4^{-k})+C\omega_f( 4^{-k}),$$
		and $A$ is a large constant to be chosen later, depending only on the dimension and ellipticity constants. 
		
		In the sequel, let us consider $\tilde u_k$ to be a solution to
		\[
		\left\{\begin{array}{rclll}
			\p_t \tilde u_k - \L \tilde u_k & = & 0 & \text{in} & \Omega \cap Q_{4^{-k}}\\
			\tilde u_k & = & |u| & \text{on} & \p_p Q_{4^{-k}} \cap \Omega\\
			\tilde u_k & = & 0 & \text{on} & \p_\Gamma\Omega \cap Q_{4^{-k}}.
		\end{array}\right.
		\]
		In particular, it follows from the induction hypothesis that
        \[
		\tilde u_k \leq c_{k-1}\varphi_{k-1} +  4^{-k+1}d_{k-1} \quad \text{on} \, \,  \p_p Q_{4^{-k}} \cap \Omega,
		\]
		and by applying the comparison principle, we have that
		\begin{equation} \label{ineq_uk}
			\tilde u_k \leq c_{k-1}\varphi_{k-1} +  4^{-k+1}d_{k-1} \quad \text{in} \, \,  \Omega \cap Q_{4^{-k}}.   
		\end{equation}

        Since $\varphi_{k-1} := \varphi_{2^{-2k+1}}$, the growth estimates in Proposition \ref{prop:special_solns} give
		\[
		\varphi_{k-1}\left(2^{-2k-1}e_n, -3\cdot 2^{-2(2k+1)}\right) \simeq 2^{-2k - 1}
		\]
		
		It follows from the inequality \eqref{ineq_uk} that
		\[
		\tilde u_k - c_{k-1}\varphi_{k-1} \le 4^{-k+1}d_{k-1} \quad \text{in} \, \,  \Omega \cap Q_{4^{-k}}.   
		\]
		Using the parabolic boundary Harnack inequality (Theorem \ref{thm:boundary_harnack}) with $\frac{ \tilde u_k - c_{k-1}\varphi_{k-1}}{4^{-k+1}d_{k-1}} $ and $\frac{\varphi_{k-1}}{\| \varphi_{k-1}\|_{L^\infty(Q_{4^{-k}})}}$, we get
		\[
		\tilde u_k - c_{k-1}\varphi_{k-1} \le C \frac{\varphi_{k-1}}{\| \varphi_{k-1}\|_{L^\infty(Q_{4^{-k}})}}4^{-k+1}d_{k-1} \leq Cd_{k-1}\varphi_{k-1} \quad \text{in} \ \Omega\cap Q_{2^{-2k-1}},
		\]
		that is, $\tilde u_k \leq (c_{k-1} + Cd_{k-1})\varphi_{k-1}$ in $\Omega\cap Q_{2^{-2k-1}}$.
		
		Now, we argue as in the proof of Theorem \ref{thm:hopf}. First, we consider the functions $v$ and $w$ defined as in Lemma \ref{Lemma_ATL2}, namely 
		$$
		v(x, t) := \frac{\varphi_{k-1}\left(2^{-2k-1}x, 2^{-2(2k+1)}t\right)}{\varphi_{k-1}\left(2^{-2k-2}e_n, -3 \cdot 2^{-2(2k+2)}  \right)}, \ \text{and} \ w(x, t) := \frac{\varphi_k\left(2^{-2k-1}x, 2^{-2(2k+1)}t\right)}{\varphi_k\left(2^{-2k-2}e_n, -3 \cdot 2^{-2(2k+2)} \right)}.
		$$
		In the sequel, we introduce the following auxiliary function 
		$$
		h := \frac{ (1+\mu_0^{-1}M\varepsilon_{k-1})w - v}{\mu_0^{-1}M\varepsilon_{k-1}}.
		$$
		Denoting by $\tilde\L$ and $\tilde\Omega$ the appropriate rescalings of $\L$ and $\Omega$, $h$ is a solution to
		$$
		\left\{\begin{array}{rclll}
			\p_t h - \tilde\L h & = & 0 & \text{in} & \tilde\Omega\cap Q_1\\
			h & \geq & \!-\mu_0 & \text{in} & \tilde\Omega\cap Q_1\\
			h & = & 0 & \text{on} & \p_\Gamma\tilde\Omega\cap Q_1\\
			h\left(\frac{e_n}{2}, -\frac{3}{4}\right)\hspace{-0.3em} & = & 1, &&
		\end{array}\right.
		$$

        It is straightforward to check that $h$ solves the above problem; in particular, Lemma \ref{Lemma_ATL2} gives $h \geq -\mu_0$. Then, by Lemma \ref{lem:almost_positivity}, $h \geq 0$ in $\tilde\Omega\cap Q_{1/2}$, that is,
		\[
		v \leq (1+\mu_0^{-1}M\varepsilon_{k-1})w\quad  \text{in} \quad \tilde\Omega\cap Q_{1/2}.
		\]
		Scaling back and computing in the same way as in the proof of the Theorem  \ref{thm:hopf}.
		\begin{align*}
			\tilde u_k & \le \left( c_{k-1} + C d_{k-1}\right) \varphi_{k-1} \\
			&\le  \left( c_{k-1} + C d_{k-1}\right)  (1+\mu_0^{-1}M\varepsilon_{k-1}) \frac{\varphi_{k-1}\left(2^{-2k-2}e_n, -3 \cdot 2^{-2(2k+2)} \right)}{\varphi_k\left(2^{-2k-2}e_n, -3 \cdot 2^{-2(2k+2)} \right)} \varphi_k\\
			&\le  \left[(1+A\varepsilon_{k-1})c_{k-1} + Ad_{k-1} \right]\varphi_k
		\end{align*}
		in $\Omega \cap Q_{4^{-k-1}}$ provided that $\varepsilon_{k-1}$ is sufficiently small. In other words, 
		\[
		\tilde u_k \le c_k \varphi_k \, \, \text{in} \, \,  \Omega \cap Q_{4^{-k-1}}.
		\]
		Arguing as in \eqref{induction_first-case}, we may apply the comparison principle and Theorem \ref{thm:ABPKT} to obtain
		$$
		|u| \leq c_k\varphi_k + 4^{-k}d_k \quad \text{in} \, \,  \Omega \cap Q_{4^{-k-1}},
		$$
		Now, we compute
		$$c_k = c_0\prod\limits_{j=0}^{k-1}(1+A\varepsilon_j) + A\sum\limits_{i=0}^{k-1}d_i\prod\limits_{j=i+1}^{k-1}(1+A\varepsilon_j) \leq c_0 e^{A\sum\limits_{j=0}^{k-1}\varepsilon_j} + A\sum\limits_{i=0}^{k-1}d_ie^{A\sum\limits_{j={i+1}}^{k-1}\varepsilon_j}.$$
		Thus, for any $r \in [4^{-k-2},4^{-k-1}]$,
		
		\begin{align*}
			\|u\|_{L^\infty(Q_r)} &\leq c_k\|\varphi_k\|_{L^\infty(Q_r)} + 4^{-k}d_k\\
			&\leq \left(c_0 e^{A\sum\limits_{j=0}^{k-1}\varepsilon_j} + A\sum\limits_{i=0}^{k-1}d_ie^{A\sum\limits_{j={i+1}}^{k-1}\varepsilon_j}\right)\|\varphi_k\|_{L^\infty(Q_r)} + 4^{-k}d_k\\
			&\leq \left(c_0 e^{A\sum\limits_{j=0}^{k-1}\varepsilon_j} + A\sum\limits_{i=0}^{k-1}d_ie^{A\sum\limits_{j={i+1}}^{k-1}\varepsilon_j}\right)(4^{-k})^{\varepsilon_k}(2r)^{1-\varepsilon_k} + 4^{-k}d_k\\
			&\leq 16\left(c_0 e^{A\sum\limits_{j=0}^{k-1}\varepsilon_j} + A\sum\limits_{i=0}^{k-1}d_ie^{A\sum\limits_{j={i+1}}^{k-1}\varepsilon_j} + d_k\right)r,
		\end{align*}		
		and by using Proposition \ref{prop:special_solns}, we may replace $\varepsilon_k= C_0[\Gamma]_{C_p^{0, 1}\left(Q'_{4^{-k}}\right)}$ to obtain
		\begin{align*}
			\frac{\|u\|_{L^\infty(Q_r)}}{r} &\leq C\left(\|u\|_{L^\infty(Q_{1/2})}\exp \left(C\int_{16r}^{2}\omega(s)\frac{\mathrm{d}s}{s} \right) + \int_{4r}^{2}[\omega_g+\omega_f](s)e^{C\left(\int_{16r}^{s/2}\omega(t)\frac{\mathrm{d}t}{t}\right)_+}\frac{\mathrm{d}s}{s}\right)\\
			&\leq C\left(\|u\|_{L^\infty(Q_1)} + \int_r^2[\omega_g+\omega_f](s)\frac{\mathrm{d}s}{s}\right)\exp \left( C\int_r^2\omega(s)\frac{\mathrm{d}s}{s} \right).
		\end{align*}
		\smallskip
        
		To complete the proof, let $(x,t),(y,s) \in \overline{\Omega}\cap Q_{1/2}$ such that the parabolic distance $d := {\rm d_p}((x,t),(y,s)) = |x-y| + |t-s|^{1/2} < \frac{r_0}{16}$. Let $d_x := \operatorname{dist}((x,t),\p_p\Omega)$, $d_y := \operatorname{dist}((y,s),\p_p\Omega)$, and assume without loss of generality that $t \geq s$. In what follows, we analyze five cases:
		\smallskip
		
		{\it Case 1.}  Assume $d_x \geq \frac{r_0}{4}$.
        \smallskip
        
        We define $u = \bar{u} + w$, where $\bar{u}$ and $w$ are solutions to the following problems.
        \[
		\left\{\begin{array}{rllll}
			\p_t\bar{u} - \L\bar{u}& = & 0 & \text{in} & Q_{r_0/8}(x, t)\\
			\bar{u}& = & u & \text{on} & \p_pQ_{r_0/8}(x, t)
		\end{array}\right.
        \]
        and
        \[
        \left\{\begin{array}{rllll}
			\p_t w - \L w& = & f & \text{in} & Q_{r_0/8}(x, t)\\
			w& = & 0 & \text{on} & \p_pQ_{r_0/8}(x, t).
		\end{array}\right.
		\]

        Then, by Theorem \ref{thm:gradient_estimate},
        $$|\bar u (x,t) - \bar u (y,s)| \leq d[\bar u]_{C^{0,1}_p(Q_{r_0/16}(x,t))} \leq Cd\|\bar u\|_{L^\infty(Q_{r_0/8}(x,t))} \leq Cd\|u\|_{L^\infty(Q_1)}.$$

        Moreover, a rescaling of Theorem \ref{thm:interior_reg_modulus} gives
        $$|w(x,t) - w(y,s)| \leq Cd\int_d^{r_0}\omega_f(r)\frac{\mathrm{d}r}{r}.$$

        Combining the estimates,
        $$|u(x,t) - u(y,s)| \leq |\bar u(x,t) - \bar u(y,s)| + |w(x,t) - w(y,s)| \leq Cd\left(\|u\|_{L^\infty(Q_1)} + \int_d^{r_0}\omega_f(r)\frac{\mathrm{d}r}{r}\right).$$

		\smallskip
		{\it Case 2.}  Assume $d_y \geq \frac{r_0}{4}$.
        \smallskip
        
        By the triangle inequality, the parabolic distance from $(y,t)$ to the boundary is at least $d_y - d \geq 3r_0/16 > r_0/8$, so $Q_{r_0/8}(y,t) \subset \Omega$. Since both $(x,t)$ and $(y,s)$ belong to $Q_{r_0/16}(y,t)$, we argue as in Case 1 to obtain
        $$|u(x,t) - u(y,s)| \leq Cd\left(\|u\|_{L^\infty(Q_1)} + \int_d^{r_0}\omega_f(r)\frac{\mathrm{d}r}{r}\right).$$

        \smallskip
		{\it Case 3.} Assume $\max\{d_x,d_y\} \leq \frac{r_0}{4}$ and $\max\{d_x,d_y\} \leq 4d$.
        \smallskip

        We first bound
        \[
		|v(x,t)| \leq Cd_x\left(\|u\|_{L^\infty(Q_1)}+\int_{d_x}^{2r_0}[\omega_g+\omega_f](s)\frac{\mathrm{d}s}{s}\right)\exp \left(C\int_{d_x}^{2r_0}\omega(s)\frac{\mathrm{d}s}{s} \right)
		\]
        and
        \[
		|v(y,s)| \leq Cd_y\left(\|u\|_{L^\infty(Q_1)}+\int_{d_y}^{2r_0}[\omega_g+\omega_f](s)\frac{\mathrm{d}s}{s}\right)\exp \left(C\int_{d_y}^{2r_0}\omega(s)\frac{\mathrm{d}s}{s} \right).
		\]
        Hence,
        \begin{align*}
        |u(x, t) - u(y, s)| &\leq |v(x,t)| + |v(y,s)| + |\nabla\!_x g(0, 0) (x-y)|\\
			                &\leq C d_x\left(\|u\|_{L^\infty(Q_1)}+\int_{d_x}^{2r_0}[\omega_g+\omega_f](s)\frac{\mathrm{d}s}{s}\right)\exp \left( C\int_{d_x}^{2r_0}\omega(s)\frac{\mathrm{d}s}{s}\right) \\
                            &\quad + C d_y\left(\|u\|_{L^\infty(Q_1)}+\int_{d_y}^{2r_0}[\omega_g+\omega_f](s)\frac{\mathrm{d}s}{s}\right)\exp \left( C\int_{d_y}^{2r_0}\omega(s)\frac{\mathrm{d}s}{s}\right) \\
                            &\quad + d|\nabla\!_x g(0, 0)|.
		\end{align*}
        Now,
        $$|\nabla\!_x g(0,0)| \leq r_0^{-1}\|g - g(0,0)\|_{L^\infty(Q_{r_0})} + \omega_g(r_0) \leq C\|u\|_{L^\infty(Q_1)} + \frac{1}{\ln 2}\int_{r_0}^{2r_0}\omega_g(s)\frac{\mathrm{d}s}{s}.$$
        Then,
        \begin{align*}
        |u(x,t) - u(y,s)| &\leq Cd_x\left(\|u\|_{L^\infty(Q_1)}+\int_{d_x}^{2r_0}[\omega_g+\omega_f](s)\frac{\mathrm{d}s}{s}\right)\exp \left(C\int_{d_x}^{2r_0}\omega(s)\frac{\mathrm{d}s}{s}\right)\\
        &\quad + Cd_y\left(\|u\|_{L^\infty(Q_1)}+\int_{d_y}^{2r_0}[\omega_g+\omega_f](s)\frac{\mathrm{d}s}{s}\right)\exp \left(C\int_{d_y}^{2r_0}\omega(s)\frac{\mathrm{d}s}{s}\right)\\
        &\quad + Cd\|u\|_{L^\infty(Q_1)} + Cd\int_{r_0}^{2r_0}\omega_g(s)\frac{\mathrm{d}s}{s}.
        \end{align*}
        Since $d_x, d_y \leq 4d$, replacing $d_x$ and $d_y$ by $d$ in the integrals is either trivial (if $d_x$ or $d_y$ is larger than $d$) or follows from Lemma \ref{lem:exp_modulus_monotonicity} (if smaller), giving
        $$|u(x,t) - u(y,s)| \leq Cd\left(\|u\|_{L^\infty(Q_1)}+\int_d^{2r_0}[\omega_g+\omega_f](s)\frac{\mathrm{d}s}{s}\right)\exp \left(C\int_d^{2r_0}\omega(s)\frac{\mathrm{d}s}{s}\right).$$

        \smallskip
		{\it Case 4.} Assume $d_x \leq \frac{r_0}{4}$, $d_x \ge d_y$ and $d \le \frac{d_x}{4}$.
        \smallskip
        
        First, let $x_* = (x',\Gamma(x',t))$ so that $(x_*, t) \in \p_\Gamma\Omega$, and note that by Lemma \ref{lem:dp_vs_dh}, $d_x \leq |x - x_*| \leq 2d_x$.
        
        Now, let $v = \bar{v} + w$ where $\bar{v}$ and $w$ solve
        \[
        \left\{\begin{array}{rllll}
            \p_t\bar{v} - \L\bar{v}& = & 0 & \text{in} & Q_{d_x/2}(x, t)\\
            \bar{v}& = & v & \text{on} & \p_pQ_{d_x/2}(x, t)
        \end{array}\right.
        \]
        and
        \[
        \left\{\begin{array}{rllll}
            \p_t w - \L w& = & f & \text{in} & Q_{d_x/2}(x, t)\\
            w& = & 0 & \text{on} & \p_pQ_{d_x/2}(x, t).
        \end{array}\right.
        \]
        Then, as in \textit{Case 1},
        $$|w(x,t) - w(y,s)| \leq Cd\int_d^{2d_x}\omega_f(r)\frac{\mathrm{d}r}{r}.$$
        On the other hand, using Theorem \ref{thm:gradient_estimate} and recalling $|x - x_*| \leq 2d_x$,
        \begin{align*}
            |\bar v(x,t) - \bar v(y,s)| &\leq d\|\bar v\|_{C^{0,1}_p(Q_{d_x/4}(x,t))} \leq C\frac{d}{d_x}\|\bar v\|_{L^\infty(Q_{d_x/2}(x,t))} \leq C\frac{d}{d_x}\|v\|_{L^\infty(Q_{5d_x/2}(x_*,t))}.
        \end{align*}
        Since $(x_*,t) \in \p_\Gamma\Omega$, the first part of the theorem gives
        $$\|v\|_{L^\infty(Q_{5d_x/2}(x_*,t))} \leq Cd_x\left(\|u\|_{L^\infty(Q_1)}+\int_{d_x}^{2r_0}[\omega_g+\omega_f](s)\frac{\mathrm{d}s}{s}\right)\exp \left(C\int_{d_x}^{2r_0}\omega(s)\frac{\mathrm{d}s}{s}\right),$$
        and therefore
        $$|\bar v(x,t) - \bar v(y,s)| \leq Cd\left(\|u\|_{L^\infty(Q_1)}+\int_{d_x}^{2r_0}[\omega_g+\omega_f](s)\frac{\mathrm{d}s}{s}\right)\exp \left(C\int_{d_x}^{2r_0}\omega(s)\frac{\mathrm{d}s}{s}\right).$$
        
        Combining the estimates for $\bar v$ and $w$, and since $d < d_x$,
        $$|v(x,t) - v(y,s)| \leq Cd\left(\|u\|_{L^\infty(Q_1)}+\int_d^{2r_0}[\omega_g+\omega_f](s)\frac{\mathrm{d}s}{s}\right)\exp \left(C\int_d^{2r_0}\omega(s)\frac{\mathrm{d}s}{s}\right).$$
        Finally, since $v = u - g(0,0) - \nabla\!_x g(0,0)\cdot x'$,
        \begin{align*}
        |u(x,t) - u(y,s)| &\leq |v(x,t) - v(y,s)| + d|\nabla\!_x g(0,0)|\\
        &\leq Cd\left(\|u\|_{L^\infty(Q_1)}+\int_d^{2r_0}[\omega_g+\omega_f](s)\frac{\mathrm{d}s}{s}\right)\exp \left(C\int_d^{2r_0}\omega(s)\frac{\mathrm{d}s}{s}\right),
        \end{align*}
        where in the last step we used the bound on $|\nabla\!_x g(0,0)|$ from \textit{Case 3}.
        
        \smallskip
		{\it Case 5.} Assume $d_y \leq \frac{r_0}{4}$, $d_y \ge d_x$ and $d \le \frac{d_y}{4}$.
        \smallskip

        By the triangle inequality, the parabolic distance from $(y,t)$ to the boundary is at least $d_y - d \geq 3d_y/4 > d_y/2$, so $Q_{d_y/2}(y,t) \subset \Omega$. Since both $(x,t)$ and $(y,s)$ belong to $Q_{d_y/4}(y,t)$, arguing as in \textit{Case 4},
        $$|u(x,t) - u(y,s)| \leq Cd\left(\|u\|_{L^\infty(Q_1)}+\int_d^{2r_0}[\omega_g+\omega_f](s)\frac{\mathrm{d}s}{s}\right)\exp \left(C\int_d^{2r_0}\omega(s)\frac{\mathrm{d}s}{s}\right).$$
		
		Finally, combining the five above cases completes the proof.
	\end{proof}

    \appendix
    \section{Proof of Theorem \ref{thm:interior_reg_modulus}}\label{sect:app}

    The Green function $G(x,t;\xi,\tau)$ for $\p_t - \mathcal{L}$ on $B_1 \times (-1,t)$ can be constructed from the fundamental solution $\Gamma(x,t;\xi,\tau)$ via standard methods, giving the representation
    $$u(x,t) = \int_{-1}^t\int_{B_1}G(x,t;\xi,\tau)f(\xi,\tau)\mathrm{d}\xi\mathrm{d}\tau.$$
    By \cite[Theorem 1.1]{ZC22}, $\Gamma$ satisfies Gaussian estimates which are inherited by $G$:
    $$|G(x,t;\xi,\tau)| \leq C|t-\tau|^{-\frac{n}{2}}\exp\left(-\frac{c|x-\xi|^2}{|t-\tau|}\right),$$
    $$|\nabla\!_xG(x,t;\xi,\tau)| \leq C|t-\tau|^{-\frac{n+1}{2}}\exp\left(-\frac{c|x-\xi|^2}{|t-\tau|}\right),$$
    and
    $$|\p_tG(x,t;\xi,\tau)| \leq C|t-\tau|^{-\frac{n}{2}-1}\exp\left(-\frac{c|x-\xi|^2}{|t-\tau|}\right),$$
    for some positive constants $C$ and $c$, depending only on the dimension, ellipticity constants and the modulus of continuity of the coefficients of $\L$.

    Now, let $r = |x - y| + |t - s|^{1/2}$, and assume without loss of generality that $t \geq s$. We split the integration domain in four regions:
    \begin{align*}
        B_1 \times (-1,t) &\subset (x,t) + \left(Q_{2r} \cup \bigcup\limits_{k = 1}^M (B_{2^{k+1}r}\setminus B_{2^kr})\times(-4r^2,0)\right.\\
        &\qquad\left.\cup \bigcup\limits_{k = 1}^M B_{2^{k+1}r}\times(-4^{k+1}r^2,-4^kr^2)\right.\\
        &\qquad\left.\cup \bigcup\limits_{k = 1}^M\bigcup\limits_{l = 1}^{k-1} (B_{2^{k+1}r}\setminus B_{2^kr})\times (-4^{l+1}r^2,-4^lr^2)\right),
    \end{align*}
    where we choose $M$ to be the minimum positive integer such that $2^{M+1}r \geq \frac{5}{4}$.

    To control the oscillation of $u$, we will compute a bound on
    $$\int|G(x,t;\xi,\tau) - G(y,s;\xi,\tau)|f(\xi,\tau)\mathrm{d}\xi\mathrm{d}\tau$$
    in each region, and then combine the estimates.

    \begin{itemize}
        \item First, in $Q_{2r}(x,t)$, we can simply bound
        $$|G(x,t;\xi,\tau) - G(y,s;\xi,\tau)| \leq |G(x,t;\xi,\tau)| + |G(y,s;\xi,\tau)|,$$
        and then
        \begin{align*}
        \int_{Q_{2r}(x,t)}|G(x,t;\xi,\tau)f(\xi,\tau)|\mathrm{d}x\xi\mathrm{d}\tau &\lesssim \int_{Q_{2r}}|\tau|^{-\frac{n}{2}}\exp\left(-\frac{c|\xi|^2}{|\tau|}\right)|f(x+\xi,t+\tau)|\mathrm{d}\xi\mathrm{d}\tau\\
        &\leq \left(\int_{Q_{2r}}|\tau|^{-\frac{n+1}{2}}\exp\left(-\frac{c(n+1)|\xi|^2}{n|\tau|}\right)\mathrm{d}\xi\mathrm{d}\tau\right)^{\frac{n}{n+1}}\\
        &\qquad\times\left(\int_{Q_{2r}}|f(x+\xi,t+\tau)|^{n+1}\mathrm{d}\xi\mathrm{d}\tau\right)^{\frac{1}{n+1}}\\
        &\lesssim \left(\int_{-4r^2}^0|\tau|^{-\frac{1}{2}}\mathrm{d}\tau\right)^{\frac{n}{n+1}}(2r)^{\frac{1}{n+1}}\omega_f(2r)\\
        &\lesssim r\omega_f(2r).
        \end{align*}
        An analogous computation gives the same bound for $G(y,s;\cdot)$.

        \item In $(B_{2^{k+1}r}(x)\setminus B_{2^kr}(x))\times(t-4r^2,t)$, we first do the following auxiliary computation. For $a > 0$ and $t \in (-4r^2,0)$,
        \begin{align*}
            \int_{B_{2^{k+1}r}\setminus B_{2^kr}}\exp\left(-a\frac{|x|^2}{|t|}\right)\mathrm{d}x &= |t|^{\frac{n}{2}}\int_{B_{2^{k+1}|t|^{-1/2}r}\setminus B_{2^k|t|^{-1/2}r}}\exp\left(-a|z|^2\right)\mathrm{d}z\\
            &\leq |t|^{\frac{n}{2}}\int_{B_{2^{k+1}|t|^{-1/2}r}}\exp\left(-a|z|^2\right)\mathrm{d}z\\
            &\leq C_{n,a}'|t|^{\frac{n}{2}}\int_{2^k}^\infty\zeta^{n-1}e^{-a\zeta^2}\mathrm{d}\zeta \leq C_{n,a}\cdot 2^{-k}|t|^{\frac{n}{2}}.
        \end{align*}
        
        Then, we also treat $G(x,t;\cdot)$ and $G(y,s;\cdot)$ separately:
        \begin{align*}
            &\int_{B_{2^{k+1}r}(x)\setminus B_{2^kr}(x)}\int_{t-4r^2}^t |G(x,t;\xi,\tau)f(\xi,\tau)|\mathrm{d}x\xi\mathrm{d}\tau\\
            &\qquad \lesssim \int_{B_{2^{k+1}r}\setminus B_{2^kr}}\int_{-4r^2}^0|\tau|^{-\frac{n}{2}}\exp\left(-\frac{c|\xi|^2}{|\tau|}\right)|f(x+\xi,t+\tau)|\mathrm{d}\xi\mathrm{d}\tau\\
            &\qquad \lesssim \left(\int_{B_{2^{k+1}r}\setminus B_{2^kr}}\int_{-4r^2}^0|\tau|^{-\frac{n+1}{2}}\exp\left(-\frac{c'|\xi|^2}{|\tau|}\right)\mathrm{d}\xi\mathrm{d}\tau\right)^{\frac{n}{n+1}}(2^{k+1}r)^{\frac{1}{n+1}}\omega_f(2^{k+1}r)\\
            &\qquad \lesssim \left(2^{-k}\int_{-4r^2}^0|\tau|^{-\frac{1}{2}}\mathrm{d}\tau\right)^{\frac{n}{n+1}}(2^{k+1}r)^{\frac{1}{n+1}}\omega_f(2^{k+1}r) \lesssim r\omega_f(2^{k+1}r).
        \end{align*}

        \item In $B_{2^{k+1}r}(x) \times (t-4^{k+1}r^2, t-4^kr^2)$, we use the mean value theorem and the estimates on $\nabla\!_xG$ and $\p_tG$ to compute
        \begin{align*}
            &\int_{B_{2^{k+1}r}(x)}\int_{t-4^{k+1}r^2}^{t-4^kr^2}\big|(G(x,t;\xi,\tau) - G(y,s;\xi,\tau))f(\xi,\tau)\big|\mathrm{d}\xi\mathrm{d}\tau\\
            &\qquad \lesssim \int_{B_{2^{k+1}r+r}}\int_{-4^{k+1}r^2-r^2}^{-4^kr^2+r^2}\left(\frac{r}{|\tau|^{\frac{1}{2}}}+\frac{r^2}{|\tau|}\right)|\tau|^{-\frac{n}{2}}\exp\left(-\frac{c|\xi|^2}{|\tau|}\right)|f(x+\xi,t+\tau)|\mathrm{d}\xi\mathrm{d}\tau\\
            &\qquad \leq 2^{-k}\left(\int_{B_{2^{k+1}r+r}}\int_{-4^{k+1}r^2-r^2}^{-4^kr^2+r^2}|\tau|^{-\frac{n+1}{2}}\exp\left(-\frac{c'|\xi|^2}{|\tau|}\right)\mathrm{d}\xi\mathrm{d}\tau\right)^{\frac{n}{n+1}}\\[0.05cm]
            &\qquad\qquad\times(2^{k+1}r + r)^{\frac{1}{n+1}}\omega_f(2^{k+1}r + r)\\
            &\qquad \lesssim 2^{-k}\left(\int_{-4^{k+1}r^2-r^2}^{4^kr^2 + r^2}|\tau|^{-\frac{1}{2}}\mathrm{d}\tau\right)^{\frac{n}{n+1}}(2^kr)^{\frac{1}{n+1}}\omega_f((2^{k+1}+1)r) \lesssim r\omega_f((2^{k+1}+1)r).
        \end{align*}

        \item Finally, an analogous argument in $(B_{2^{k+1}r}\setminus B_{2^kr}) \times (-4^{l+1}r^2,-4^lr^2)$ gives (note that $r^2|\tau|^{-1}<r|\tau|^{-\frac{1}{2}} < 1$):
        \begin{align*}
            &\int_{B_{2^{k+1}r}(x) \setminus B_{2^kr}(x)}\int_{t-4^{l+1}r^2}^{t-4^lr^2}\big|(G(x,t;\xi,\tau) - G(y,s;\xi,\tau))f(\xi,\tau)\big|\mathrm{d}\xi\mathrm{d}\tau\\
            &\qquad \lesssim \int_{B_{2^{k+1}r+r} \setminus B_{2^kr - r}}\int_{-4^{l+1}r^2-r^2}^{-4^lr^2+r^2}\frac{r}{|\tau|^{\frac{1}{2}}}|\tau|^{-\frac{n}{2}}\exp\left(-\frac{c|\xi|^2}{|\tau|}\right)|f(x+\xi,t+\tau)|\mathrm{d}\xi\mathrm{d}\tau\\
            &\qquad \lesssim 2^{-l}(2^lr)^{-n}\exp\left(-c'\cdot 4^{k-l}\right)\int_{B_{2^{k+1}r+r} \setminus B_{2^kr - r}}\int_{-4^{l+1}r^2-r^2}^{-4^lr^2+r^2}|f(x+\xi,t+\tau)|\mathrm{d}\xi\mathrm{d}\tau\\
            &\qquad \lesssim 2^{-l}(2^lr)^{-n}\exp\left(-c'\cdot 4^{k-l}\right)\left((2^kr)^n(4^lr^2)\right)^{\frac{n}{n+1}}(2^{k+1}r+r)^{\frac{1}{n+1}}\omega_f(2^{k+1}r+r)\\
            &\qquad \lesssim 2^{-l-nl+(kn+2l)\frac{n}{n+1}+\frac{k}{n+1}}r^{-n+(n+2)\frac{n}{n+1}+\frac{1}{n+1}}\exp\left(-c'\cdot 4^{k-l}\right)\omega_f((2^{k+1}+1)r)\\
            &\qquad = 2^{(k-l)\frac{n^2+1}{n+1}}\exp\left(-c'\cdot 4^{k-l}\right)r\omega_f((2^{k+1}+1)r).
        \end{align*}
    \end{itemize}

    Now, combining the estimates in all regions,
    \begin{align*}
        |u(x,t) - u(y,s)| &\lesssim r\omega_f(2r) + \sum\limits_{k=1}^M r\omega_f(2^{k+1}r) + \sum\limits_{k = 1}^M r\omega_f((2^{k+1}+1)r)\\
        &\qquad+\sum\limits_{k=1}^M\sum\limits_{l=1}^{k-1}2^{\frac{n^2+1}{n+1}(k - l)}e^{-c'\cdot 4^{k-l}}r\omega_f((2^{k+1}+1)r)\\
        &\lesssim r\sum\limits_{k=1}^M\left[\omega_f((2^{k+1}+1)r)\sum\limits_{m = 1}^{k - 1}2^{\frac{n^2+1}{n+1}m}e^{-c'\cdot 4^m}\right]\\
        &\leq r\sum\limits_{k=1}^M\omega_f((2^{k+1}+1)r)\sum\limits_{m \geq 1}2^{\frac{n^2+1}{n+1}m}e^{-c'\cdot 4^m}\\
        &\lesssim r\sum\limits_{k=1}^M\omega_f((2^{k+1}+1)r) \lesssim r\int_r^{4}\omega_f(\rho)\frac{\mathrm{d}\rho}{\rho},
    \end{align*}
    where we used that $(2^{M+1}+1)r < 3$.

	\textbf{Data availability statement.}
	Data sharing not applicable to this article as no datasets were generated or analysed during the current study.
	
	\textbf{Conflict of interest statement.}
	The authors have no competing interests to declare that are relevant to the content of this article.


\begin{thebibliography}{MMW17}
    
		
		\bibitem[AN22]{AN22} D. Apushkinskaya, A. Nazarov, \textit{The normal derivative lemma and surrounding issues}, Russian Math. Surveys \textbf{77} (2022), 189–249.
		
		\bibitem[CC95]{CC95} L. Caffarelli, X. Cabré, \textit{Fully Nonlinear Elliptic Equations}, Amer. Math. Soc. Colloq. Publ. \textbf{43} (1995).
		
		
		\bibitem[CSV18]{CSV18} M. Colombo, L. Spolaor, B. Velichkov, \textit{A logarithmic epiperimetric inequality for the obstacle problem}, Geom. Funct. Anal. \textbf{28} (2018), 1029–1061.
		
		\bibitem[CSV20]{CSV20} M. Colombo, L. Spolaor, B. Velichkov, \textit{Direct epiperimetric inequalities for the thin obstacle problem and applications}, Comm. Pure Appl. Math. \textbf{73} (2020), 384-420.

        \bibitem[CKS00]{CKS00} M. Crandall, M. Kocan, A. Swiech, \textit{$L^p$ theory for fully nonlinear uniformly parabolic equations}, Comm. Partial Differential Equations \textbf{25} (2000), 1997-2053.

        \bibitem[DKM14]{DKM14} P. Daskalopoulos, T. Kuusi, G. Mingione, \textit{Borderline estimates for fully nonlinear elliptic equations}, Comm. Partial Differential Equations \textbf{39} (2014), 574–590.
        
		
		\bibitem[DS22]{DS22} D. De Silva, O. Savin, \textit{On the parabolic boundary {H}arnack principle}, La Matematica \textbf{1} (2022), 1-18.
		

        \bibitem[DLL25]{DLL25} J. Dong, X. Li, Y. Lian, \textit{Boundary Regularity for Fully Nonlinear Parabolic equations on $C^{1,\mathrm{Dini}}$ Domains}, preprint arXiv (2025).

		\bibitem[FR22]{FR22} X. Fernández-Real, X. Ros-Oton, \textit{Regularity Theory for Elliptic PDE}, Zurich Lectures in Advanced Mathematics, EMS Press (2022).
		
        \bibitem[FRT26]{FRT26} G. Fioravanti, X. Ros-Oton, C. Torres-Latorre, \textit{Extinction rates for nonradial solutions to the Stefan problem}, Trans. Amer. Math. Soc., in press (2026).

       \bibitem[Fri58]{Fri58} A. Friedman, \textit{Remarks on the maximum principle for parabolic equations and its applications}, Pacific J. Math. \textbf{8} (1958), 201-211.
		
		\bibitem[HR19]{HR19} M. Hadžić, P. Raphaël, \textit{On melting and freezing for the 2D radial Stefan problem}, J. Eur. Math. Soc. \textbf{21} (2019), 3259–3341.
		

        \bibitem[Kam88]{Kam88} L. Kamynin, \textit{A theorem on the interior derivative for a second-order uniformly parabolic equation}, Dokl. Akad. Nauk SSSR \textbf{299} (1988), 280-283 (in Russian). English transl. in Sov. Math. Dokl. \textbf{37} (1988), 373-376.

        \bibitem[KK73]{KK73} L. Kamynin, B. Khimchenko, \textit{On Giraud's theorem for a second-order parabolic equation}, Sibirsk. Mat. Zh. \textbf{14} (1973), 86-96 (in Russian). English transl. in Sib. Math. J. \textbf{14} (1973), 59-66.
         
        \bibitem[KK74]{KK74} L. Kamynin, B. Khimchenko, \textit{On the maximum principle and Lipschitz boundary estimates for the solution of a second-order parabolic equation}, Sibirsk. Mat. Zh. \textbf{15} (1974), 343-367 (in Russian). English transl. in Sib. Math. J. \textbf{15} (1974), 242-260.
         
        \bibitem[KK75]{KK75} L. Kamynin, B. Khimchenko, \textit{On local Lipschitz estimates of solutions of a second-order parabolic equation near the lateral boundary of the domain}, Sibirsk. Mat. Zh. \textbf{16} (1975), 1264-1282 (in Russian). English transl. in Sib. Math. J. \textbf{16} (1975), 897-909.
        
		
		\bibitem[Kry76]{Kry76} N. Krylov, \textit{Sequences of convex functions and estimates of the maximum of the solution of a parabolic equation}, Sib. Math. J. \textbf{17} (1976), 226-236.
		
		\bibitem[LU88]{LU88} O. Ladyzhenskaya, N. Ural'tseva, \textit{Estimates on the boundary of the domain of first derivatives of functions satisfying an elliptic or a parabolic inequality}, Boundary value problems of mathematical physics. 13, Tr. Mat. Inst. Steklova \textbf{179}, Nauka, Moscow 1988, pp. 102–125 (in Russian). English transl. in Proc. Steklov Inst. Math. \textbf{179} (1989), 109-135.
		
		
		\bibitem[Lie85]{Lie85} G. Lieberman, \textit{Regularized distance and its applications}, Pacific J. Math. \textbf{117} (1985), 329–352.
		
		
		\bibitem[MMW17]{MMW17} F. Ma, D. Moreira, L. Wang, \textit{Differentiability at lateral boundary for fully nonlinear parabolic equations}, J. Differential Equations \textbf{263} (2017), 2672-2686.
		
		
		\bibitem[Naz12]{Naz12} A. Nazarov, \textit{A centennial of the Zaremba-Hopf-Oleinik lemma}, SIAM J. Math. Anal. \textbf{44} (2012), 437-453.
		
       
       \bibitem[Nir53]{Nir53} L. Nirenberg, \textit{A strong maximum principle for parabolic equations}, Comm. Pure Appl. Math.  \textbf{6} (1953), 167-177.
		

        \bibitem[Puc66]{Puc66} C. Pucci, \textit{Operatori ellittici estremanti},  Ann. Mat. Pura Appl. (4) \textbf{72} (1966), 141 -170.

        \bibitem[Tei14]{Tei14} E. Teixeira, \textit{Universal moduli of continuity for solutions to fully nonlinear elliptic equations}, Arch. Rational Mech. Anal. \textbf{211} (2014), 911–927.
        
		\bibitem[Tor24]{Tor24} C. Torres-Latorre, \textit{Parabolic boundary Harnack inequalities with right-hand side}, Arch. Rational Mech. Anal. \textbf{248} (2024), 73.

        \bibitem[Tor26]{Tor26} C. Torres-Latorre, \textit{Boundary estimates for non-divergence equations in $C^1$ domains}, Calc. Var. Partial Differential Equations \textbf{65} (2026), 11.
		
		\bibitem[Tso85]{Tso85} K. Tso, \textit{On an Aleksandrov-Bakel'Man type maximum principle for second-order parabolic equations}, Comm. Partial Differential Equations \textbf{10} (1985), 543-553.

       \bibitem[Vyb57]{Vyb57} R.  {V\'yborn\'y}, \textit{On the properties of the solutions of some boundary problems for equations of parabolic type}, Dokl. Akad. Nauk SSSR, \textbf{117},  (1957), 563-565 (in Russian).

        \bibitem[ZC22]{ZC22} I. Zhenyakova, M. Cherepova, \textit{The Cauchy problem for a multi-dimensional parabolic equation with Dini-continuous coefficients}, J. Math. Sci. \textbf{264} (2022), 581-602.
		
	\end{thebibliography}
\end{document}